\documentclass[11pt,reqno]{amsart}

\usepackage{verbatim}
\usepackage[dvips]{graphicx}
\usepackage{epsfig,subfigure}
\usepackage{psfrag}
\usepackage{amsmath,amssymb,amsfonts,bm,amscd}

\numberwithin{equation}{section}

\newtheorem{theorem}{Theorem}[section]
\newtheorem{corollary}[theorem]{Corollary}
\newtheorem{lemma}[theorem]{Lemma}
\newtheorem{proposition}[theorem]{Proposition}

\theoremstyle{definition}
\newtheorem{definition}{Definition}[section]
\newtheorem{example}{Example}[section]
\newtheorem{assum}{Assumption}[section]
\theoremstyle{remark}

\newtheorem{problem}{Problem}[section]

\newtheorem{algorithm}{Algorithm}

\newcommand{\scale}{0.8}
\newcommand{\ri}[1]{\stackrel{o}{#1}{}}
\newcommand{\rk}{\text{rank}}

\newcommand{\RR}{{\mathbb R}}
\newcommand{\cA}{{\mathcal A}}
\newcommand{\cB}{{\mathcal B}}
\newcommand{\cF}{{\mathcal F}}
\newcommand{\cG}{{\mathcal G}}
\newcommand{\cH}{{\mathcal H}}
\newcommand{\cN}{{\mathcal N}}
\newcommand{\cO}{{\mathcal O}}
\newcommand{\cP}{{\mathcal P}}
\newcommand{\PP}{{\mathbb P}}
\newcommand{\cQ}{{\mathcal Q}}
\newcommand{\cR}{{\mathcal R}}
\newcommand{\cS}{{\mathcal S}}

\newcommand{\reach}{{\text{Reach}}}
\newcommand{\conv}{{\text{conv}}}

\begin{document}
\title[Triangulation and Control]{Triangulations, Subdivisions, and Covers for
Control of Affine Hypersurface Systems on Polytopes}
\author{Zhiyun Lin}
\email{linz@zju.edu.cn}
\author{Mireille~E.~Broucke}
\address{Department of Electrical and Computer Engineering,
        University of Toronto, Toronto, ON M5S 3G4, Canada}
\email{broucke@control.utoronto.ca}
\date{\today}

\begin{abstract}
This paper studies the problem for an affine hypersurface system to reach a polytopic target set
starting from inside a polytope in the state space.
We present an exhaustive solution which begins with a 
characterization of states which can reach the target by open-loop control and concludes
with a systematic procedure to synthesize a feedback control. Our emphasis is
on methods of subdivision, triangulation, and covers which explicitly account for
the capabilities of the control system. In contrast with previous literature, 
the partition methods are guaranteed to
yield a correct feedback synthesis, assuming the problem is solvable by open-loop control.
\end{abstract}

\maketitle

\section{Introduction}

Problems of reachability for dynamical systems
have been extensively studied in the control literature for a long
time. These problems have attracted renewed interest due to the
emergence of new paradigms for switched and piecewise
linear systems. This paper studies the problem for an affine system to reach a polytopic target set
starting from inside a polytope in the state space.
Promising new ideas have appeared in the last five years in this area, and these ideas have stimulated deeper
study of the many open questions that remain. An important gap in the literature is an exhaustive solution
which covers the following sub-problems: explicit conditions for and an analysis of all states which can reach the target by open-loop control;
a method to approximate the open-loop reachable states when there are control constraints;
a systematic method to form a subdivision of the polytope
into a set of reachable states and a set of failure states;
and finally a systematic procedure to synthesize a feedback control. While
parts of this research program have been studied under various assumptions, no overall end-to-end solution has been presented.
The reason is that the problem is generally extremely difficult and for certain steps, almost
nothing is known about systematic procedures. In order to tackle this problem, rather than scoping back the problem
specification as has typically been done before, we retain the complete problem statement but work with
a specific class of systems: affine hypersurface systems which are $n$-dimensional affine systems
with $(n-1)$ inputs. This class is our focus of study for the following reasons:
(1) The problem of developing a systematic methodology to synthesize controllers for reachability specifications is essentially
open, and beginning with a specific class of models provides much needed insight which can be built upon for generalization.
The outcome of our study is that we are able to provide a complete solution for affine hypersurface systems.
In so doing we introduce new techniques for triangulation and subdivision which can be adapted to the general problem.
(2)
Hypersurface systems include as a special case second-order mechanical systems, which 
are an important benchmark class for new control design methods. More generally,
second-order systems have attracted extensive theoretical study due to their strong geometric properties.
(See \cite{BOSCAIN} for a recent example).
(3)
Hypersurface systems have particularly simple reachable sets. By studying these systems, we separate the
challenges inherent in dealing with complex reachable sets from the other challenges presented by dealing with control
synthesis on a state space which is a polytope. The contributions of the paper are therefore squarely in the area of
triangulations, subdivisions, and covers. What this suggests for future investigations is very important: if the designer is
willing to relax the requirement to find the {\em largest} set of states in the polytope that can reach the target and instead 
he works with
approximations which have reasonable properties, and where importantly, {\em reasonable properties are determined not based on traditional
interpretations of hard and easy reachability computations, but based on how easily one can find triangulations and subdivisions to
solve the synthesis problem}, then one has a hope to develop systematic procedures which are provably successful.

We will now outline the sub-problems which are addressed in this
paper. Some of these sub-problems simply involve packaging known
results in an appropriate way. Other results are novel and have
never been studied before. The latter are especially in the area
of forming triangulations and subdvisions adapted to a given
control synthesis problem. The first sub-problem we address is:
given a polytopic state space $\cP$ and a polytopic target set
$\cF$ in the boundary of $\cP$, characterize explicitly the set of
states in $\cP$ which can reach the target by open-loop control.
This result relies on well-known properties of the reachable sets
of hypersurface systems. While it is a stepping stone to later
results, its importance stems from the fact that knowing
explicitly if a particular reachability problem is solvable by
open-loop control gives a concrete metric against which to test
our results: our synthesis methods should apply to any problem for
which is a solution by open-loop control exists. The next
sub-problem is to develop an algorithm which ``cuts off'' the
failure states in a systematic way, so that the remaining set
$\cP'$ contains the original target set $\cF$ and is a polytope
for which the open-loop reachability problem is solvable. It is
shown that a systematic method to cut off the failure regions can
be done with only two techniques based on the system structure.
This algorithm can be easily adapted to include bounds on the
control input. Once it is known that for a polytopic subset of the
state space all points can reach the target by open-loop control, 
one then addresses questions of control synthesis, and this is the
heart of the paper.

We develop a set of triangulation and subdivision procedures which are organized hierarchically. By a hierarchical
organization we mean the following. At level one of the hierarchy is a subdivision method which solves the given
reachability problem on the polytope. When that subdivision method is
applied, it leads to sub-reachability problems on sub-polytopes which are solved by subdivision methods at level two of the hierarchy, and
so forth. What is important is that there are a finite number of levels and one can prove that the refinement by new subdivisions
terminates.  This contrasts with the view that one simply refines by arbitrary subdivisions selected by a computer program -
a method that has no guarantee to terminate even of the problem is solvable by open-loop control.

One of the challenges in developing these subdivision methods is to determine the simplest set of methods which can completely
solve the problem. We present such a set, though it is by no means unique. We propose a five-level hierarchy:
\begin{enumerate}
\item
For linear and affine systems, problems of reachability are closely tied to existence of equilibria. Therefore, the first level
is a subdivision along the hyperplane $\cO$ of the possible equilibria of the system; namely those points in the state space
for which there exists a control input to make the vector field exactly zero. 
\item
The second level of the hierarchy is a subdivision of a polytope which has two possible target sets: a neighboring polytope
and a target facet. The subdivision is in fact a {\em cover} to be discussed further below, and is with respect to the 
boundary of the two reachability sets.
\item
The third level entails a subdivision with respect to $\cF$ the target set. In particular, it applies to the case
when $\cF$ is not a facet of $\cP$ (motivation for this is discussed in Section~\ref{intro:motivate}). Two techniques are presented: one
is a subdivision and the other is a cover. 
\item
The fourth subdivision is a triangulation within a polytope whose interior does not intersect $\cO$ and it has a single target facet. 
The triangulation is determined using information about the dynamics on that polytope.
\item
The fifth subdivision is a triangulation within a simplex whose interior does not intersect $\cO$ and it has a single target facet. 
\end{enumerate}

In the remainder of the introduction, we review the relevant literature on control synthesis for reachability problems
on simplices and polytopes. In Section~\ref{intro:motivate} we give the context of the problem and from this arises the motivation
and characteristics of our solution. Section~\ref{intro:org} presents notation and the organization of the remainder of the paper.

\subsection{Historical Overview and Related Literature}
\label{intro:literature}
While control problems of reaching target sets in the state space have been studied since the 1960's,
our formulation and approach arise from more recent investigations on
affine and piecewise affine systems defined on simplices and polytopes. The first problem to be studied
of this type was by Luc Habets and Jan van Schuppen \cite{HVS01} in which they formulated the so-called
control-to-facet problem. Further results were given in \cite{HVS04}. Given an affine system,
the problem is to synthesize an affine control to reach an exit facet of a simplex in finite time.
Necessary conditions called {\em invariance conditions} in the form of linear inequalities
defined at the vertices of the simplex were presented
which restrict the closed-loop vector field to point inside appropriately defined tangent cones of the simplex.
A sufficient condition was also presented to ensure that all closed-loop
trajectories exit the simplex. Based on the invariance
conditions, an elegant synthesis method was proposed to obtain an affine control $u = K x + g$ to solve the problem.
In \cite{HVS06,RB06} the control-to-facet problem using affine controls for affine systems defined on simplices
was improved (to allow that trajectories need not exit the simplex at the first time they reach the exit facet),
and more concise necessary and sufficient conditions were obtained.
The new conditions consist of the original
invariance conditions of \cite{HVS01,HVS04} combined with a {\em flow condition} which guarantees
that all trajectories exit the simplex, or equivalently that the closed-loop system has
no equilibria in the simplex. The control-to-facet problem for a polytope as well as hybrid systems was also studied in
\cite{HVS06}. The proposed method is to partition the state space into simplices,
to form a discrete graph capturing the adjacency of simplices, and
then to solve, via a dynamic programming algorithm, a sequence of control-to-facet problems. When the algorithm
terminates successfully it is guaranteed to provide a piecewise affine controller solving the reachability problem.

The problem of reachability with state constraints is
related to the viability/capturability problem in viability
theory, especially characterizing viability kernels with a
target and viable-capture basins for differential inclusions. The
concept of viability kernel with a target by a Lipschitz
set-valued map has been introduce and studied in \cite{QuiVel98}:
This is the subset of initial states in a constrained set from
which at least one solution remains in the constrained set (i.e.,
is viable) forever or reaches (i.e., captures) a target in finite
time before possibly violating the constraints. The set of initial
states satisfying only the latter condition is called the
viable-capture basin of the target. Some abstract properties and
characterizations of viability kernels and viable-capture basins
of a target are further studied and provided in \cite{Aub01}.

A number of methods to construct piecewise affine feedbacks on polytopes
for various control specifications such as stabilization, optimal control, and set invariance have already been developed.
The recent text \cite{BLANCHINI} presents an overview of methods for set invariance, which can be viewed
as the dual to the problem of reachability.
Piecewise affine systems have been the subject of a large number of papers.
A small sampling of recent papers includes \cite{BARIC,BEMPORAD,BLANCHINI2,GEYER,ROLL}.
Several interesting applications of piecewise affine modeling have recently been explored. See
for example \cite{FARCOT}.

\subsection{Context and Motivation}
\label{intro:motivate}

This paper considers the problem of reaching a target $\mathcal{X}_f$ with
state constraint in a set $\mathcal{X}$, denoted as $\mathcal{X}
\overset{\mathcal{X}}{\longrightarrow} \mathcal{X}_f$. The
motivation for this fundamental problem arises from a family of
related reachability problems. Two sample reachability problems
are as follows.
\begin{enumerate}
\item {\em Reach - Avoid problem}. \
Starting at any initial state in a bounded set $Q$, reach a target set
$Q_t$ while avoiding an unsafe region $Q_u$. The problem can be formulated as
$\mathcal{X} \overset{\mathcal{X}}{\longrightarrow} Q_t$, where $\mathcal{X}=Q-Q_u$.
A typical example of the problem is motion planning of multiple
vehicles. \item  {\em Temporal Logic Controller Synthesis}. \
Consider, for example, three areas of interest denoted by $Q_0,
Q_1, Q_2$ such  that $Q_1, Q_2 \subset Q_0$ (see
Figure~\ref{Fig:LTL}) and the temporal logic specification
$\square Q_0 \wedge  \diamondsuit (Q_1 \wedge (Q_1 \mathcal{U}
Q_2) )$, which is interpreted in natural language as: ``Stay
always in $Q_0$ and visit $Q_1$, then stay in $Q_1$ until it
visits $Q_2$ eventually.'' The problem can be thought of as two
reachability problems $Q_0-Q_2 \overset{Q_0-Q_2}{\longrightarrow}
Q_1$ and $Q_1 \overset{Q_1}{\longrightarrow} Q_2$.
\end{enumerate}
\begin{figure}[!t]
\begin{center}
\psfrag{q0}{$Q_0$}
\psfrag{q1}{$Q_1$}
\psfrag{q2}{$Q_2$}
\includegraphics[width= 0.3 \linewidth]{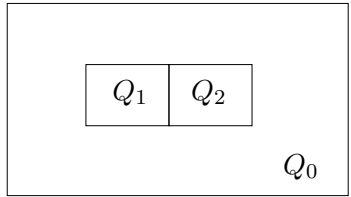}
\caption{The problem of temporal logic controller synthesis.
\label{Fig:LTL}}
\end{center}
\end{figure}
This family of reachability problems motivates the particular features of the problem studied in the paper, in which each $Q_i$ is a polytope
and each target is a polytope in the boundary of $Q_i$. The dynamics in each $Q_i$ may or may not be the same (although we do not study
the hybrid problem here). Sub-reachability problems are sequenced in order to achieve a global specification.
It may happen that a certain reachability problem fails for a particular polytope and one must restrict the polytope by cutting off failure
regions. Such restrictions would propagate to neighboring polytopes and reduce their feasible target sets.
It would be extremely tedious to leave these interventions to the designer, and rather, an automated algorithm
should resolve these failures. This justifies our choice to solve the reachability problem when the target
is not a facet of a polytope.

\subsection{Notation and Organization}
\label{intro:org}

We use the following  notation.
Let $\text{rank}(B)$ and $\text{Im}(B)$ denote the rank and the image
of a matrix $B$. Let $\mathcal{A}$ be a set. $\ri{\mathcal{A}}$,
$\text{conv}(\mathcal{A})$, $\text{vert}(\mathcal{A})$,  and
$\text{aff}(\mathcal{A})$ denote the interior of $\mathcal{A}$,
the convex hull of $\mathcal{A}$, the vertices of $\mathcal{A}$,
and the smallest affine space containing $\mathcal{A}$,
respectively. Let $\mathcal{B}$ be another set. $\mathcal{A}
\setminus \mathcal{B}$ expresses the set difference. Moreover,
$\text{dist}(x, \mathcal{A})$ expresses the distance from a point
$x$ to the set $\mathcal{A}$. Finally, we let
$\mu\left[\mathcal{C}\right] $ be the volume of an $n$-dimensional
set $\mathcal{C}$. If $\mathcal{C}$ is of dimension less than $n$
then $\mu\left[\mathcal{C}\right]=0$.

The paper is organized as follows. In Section~\ref{sec:formulate}
we formulate the problems to be solved. In
Section~\ref{sec:reachpoly} we characterize the set of states
which can reach the target $\cF$ starting in a polytope $\cP$.
Then in Sections~\ref{sec:simplices} and \ref{sec:controlpoly} we
show how to synthesize piecewise affine controls on simplices and
polytopes, respectively, assuming the problem is solvable by open
loop control. In Section~\ref{sec:trianglewrtF} we show how to
subdivide a polytope with respect to a target set which is not a
facet, and in Section~\ref{sec:trianglewrtO} we show how to
subdivide a polytope with respect to the set of possible
equilibria of the system.

\section{Problem Formulation}
\label{sec:formulate}

Let $\cP$ be an $n$-dimensional polytope in $\mathbb{R}^n$ and
consider an affine control system on $\cP$,
\begin{equation}
\label{eq:AffineSystem}
\Sigma: \quad \dot x=Ax+a+Bu =: f(x, u), \quad x \in \cP \,,
\end{equation}
where $A \in \mathbb{R}^{n \times n}$, $B \in \mathbb{R}^{n \times m}$, $a \in \mathbb{R}^n$
and the control $u \in \mathbb{R}^m$ lives in the space of piecewise continuous functions.
Assuming that $\rk(B) =n-1$, we call $\Sigma$ an {\em affine hypersurface system}.
Given a piecewise continuous function $u: t \mapsto u(t)$ and
an initial state $x_0 \in \cP$, let $\phi^u_t(x_0)$
denote the unique solution of $\Sigma$ starting from $x_0$.

In order to precisely formulate our problem, we begin by defining two concepts.
The first is that of reaching a target with constraint in a set, which is
the analogue to the notion of capturability in viability theory
\cite{Aub01}. Second, we define $\Omega$-invariant sets. In viability
theory such a set is called locally invariant relative to $\Omega$ \cite{Aub01}.

\begin{definition} \
Let $\Omega $ and $\Omega_f$ be closed sets satisfying $\Omega \cap \Omega_f \neq \emptyset$.
\begin{enumerate}
\item[(a)] \label{Def:PointBR}
A point  $x \in \Omega$ can {\em
reach $\Omega_f$ with constraint in $\Omega$}, denoted by $x
\overset{\Omega}{\longrightarrow} \Omega_f$, if there exists a
piecewise continuous control $u: t \mapsto u(t)$ and $T \ge 0$ satisfying $\phi_T^u(x)
\in \Omega_f$ and $\phi_t^u(x) \in \Omega$ for all $t \in [0, T]$.
Otherwise, we say $x$ cannot reach $\Omega_f$ with constraint in
$\Omega$, denoted by $x \not \overset{\Omega}{ \longrightarrow }
\Omega_f$.
\item[(b)] \label{Def:SetBR}
A set  $\Omega' \subseteq
\Omega$ can {\em reach $\Omega_f$ with constraint in $\Omega$},
denoted by $\Omega' \overset{\Omega}{\longrightarrow} \Omega_f$,
if $x \overset{\Omega}{\longrightarrow} \Omega_f$ for every $x \in
\Omega'$.
\end{enumerate}
\end{definition}

The {\em maximal reachable set of $\Omega_f$ in $\Omega$} will be denoted by
$\reach( \Omega, \Omega_f)$.

\begin{definition}\label{Def:InvSet}
For a closed set $\Omega$, a set $\mathcal{A} \subseteq \Omega$ is
called {\em $\Omega$-invariant} if for all $x_0 \in \mathcal{A}$
and for all piecewise continuous functions $u: t \mapsto u(t)$,
every trajectory $\phi_t^u(x_0)$ in $\Omega$ on an interval $[0, T]$ with $T < \infty$
or $[0, \infty)$ is in $\mathcal{A}$ on the same time interval.
\end{definition}

The definition means that the
trajectories cannot leave $\mathcal{A}$ before leaving $\Omega$.
The following are elementary properties of $\Omega$-invariant sets which
can be obtained directly from the definition.
\begin{lemma} \
\label{lem:elem}
The union and intersection of two $\Omega$-invariant
sets are also $\Omega$-invariant sets.
The union of all
points $x_0 \in \mathcal{A}$ for which each trajectory  segment
$\phi_t^u(x_0)$ is in $\mathcal{A}$ for the same time it is in
$\Omega$ is the maximal $\Omega$-invariant set in $\mathcal{A}$.
\end{lemma}

The following result relates the set of states that can reach $\Omega_f$ with constraint in $\Omega$ to
$\Omega$-invariant sets. The proof is in the Appendix.

\begin{proposition}
\label{Thm:NS}
$\Omega \overset{\Omega}{\longrightarrow} \Omega_f $ if and only
if no $\Omega$-invariant set is in $\Omega \setminus \Omega_f$.
\end{proposition}

Now we introduce the assumptions on $\Sigma$.
Let $\mathcal{B}$ denote the $(n-1)$-dimensional  subspace spanned by the column vectors of
$B$ (namely, $\mathcal{B}=\text{Im}(B)$, the image of $B$).
Define
\[
\mathcal{O}:= \{x \in \mathbb{R}^n ~:~ Ax+a \in \mathcal{B} \}.
\]
When the pair $(A,B)$ is controllable it can be shown that $\cO$ is an affine space (see also
\cite{NenFre02}).
Notice that $f(x,u)$ on $\mathcal{O}$ can vanish for an appropriate
choice of $u$, so $\mathcal{O}$ is the set of all possible
equilibrium points of the system.
We make the following standing assumptions until Section~\ref{sec:trianglewrtO}.
\begin{assum} \hspace{1in}
\label{assum1}
\begin{itemize}
\item[(A1)]
$\rk(B) =n-1$.
\item[(A2)]
The pair $(A,B)$ is controllable.
\item[(A3)]
$\ri{\mathcal{P}}~\cap \mathcal{O} =\emptyset$.
\item[(A4)]
The target set $\mathcal{F}$ is an $(n-1)$-dimensional polytope on the boundary of $\cP$.
\end{itemize}
\end{assum}

\begin{problem}
\label{prob1}
We are given $\Sigma$ such that Assumption~\ref{assum1} holds.
\begin{itemize}
\item[(a)]
Find necessary and sufficient conditions such that $\cP \overset{\cP}{\longrightarrow} \cF$.
\item[(b)]
Find $\text{Reach}(\cP,\cF)$, the maximal reachable set of $\cF$ in $\cP$.
\item[(c)]
Find a triangulation $\mathbb{T}$ and a piecewise affine feedback such that
$\text{Reach}(\cP,\cF) \overset{\cP}{\longrightarrow} \cF$.
\end{itemize}
\end{problem}

\section{Reachability on Polytopes \label{Sec:ReachPolytope}}
\label{sec:reachpoly}
In this section, we focus on the open-loop reachability problem
and the first aim is to find necessary and sufficient conditions for
$\cP \overset{\cP}{\longrightarrow} \cF$. The strategy is
to isolate all $\cP$-invariant sets in $\cP \setminus \cF$.
Note that proofs of lemmas for this part can either be found in the Appendix
or we have omitted them in case they were direct logic arguments not adding
insight for the reader.

Denote by $\mathcal{B}_x$ the hyperplane parallel to $\mathcal{B}$ and going
through a point $x$. Let $\beta$ be the unit normal vector to
$\mathcal{B}$ satisfying $\beta^T (Ax+a) \leq 0$ for all $x \in
\mathcal{P}$. Such $\beta$ always exists by our assumption that
$\ri{\mathcal{P}} \cap \mathcal{O} =\emptyset$. Let
$v^-$ be a point in $\arg \min \{\beta^Tx: x \in \mathcal{F}\}$ and
$v^+$ a point in $\arg \max \{\beta^Tx: x \in \mathcal{F}\}$.
Define the closed half-spaces in $\cP$
\begin{eqnarray*}
\mathcal{H}^- &:=& \{x \in \cP ~|~ \beta^T  x \leq \beta^T  v^- \} \,, \\
\mathcal{H}^+ &:=& \{x \in \cP ~|~ \beta^T  x \geq \beta^T  v^+ \} \,.
\end{eqnarray*}
Also, for any $z \in \mathbb{R}^n$, define
$\mathcal{H}^-(z) := \{x \in \cP ~|~ \beta^T  x \leq \beta^T  z \}$ and
$\mathcal{H}^+(z) := \{x \in \cP ~|~ \beta^T  x \geq \beta^T  z \}$.
Finally, we introduce the set
\[
\cP^+ := \arg \max \{\beta^Tx ~|~ x \in \cP \} \,.
\]
Because $\ri{\cP} \cap \cO = \emptyset$, we know that for each initial condition in $\cP$, all
trajectories will only flow in one direction relative to hyperplane $\cB$. In particular,
the $\beta$-component of any trajectory, $\beta \cdot \phi_t^u(x_0)$, is always non-increasing by the convention that
$\beta \cdot (Ax+a) \le 0$, $\forall x \in \cP$. Now the points $v^-$ and $v^+$ mark the points in $\cF$ with
minimum and maximum $\beta$ components. It is clear that if there is any $x_0 \in \cP$ with a $\beta$ component smaller
than $v^-$, then no $\phi_t^u(x_0)$ can reach $\cF$. The following lemma confirms this intuition by showing that
$\cH^-$ and $\ri{\cH}^-$  are $\cP$-invariant sets.

\begin{lemma}
\label{lem1}
Let $z$ be a point in $\cP$.
The sets $\cH^-(z)$ and $\ri{\cH}^-(z)$  are $\cP$-invariant.
\end{lemma}

\begin{figure}[t!]
\begin{center}
\psfrag{v1}{$v^-$}
\psfrag{v2}{$v^+$}
\psfrag{p}{$\mathcal{P}$}
\psfrag{f}{$\mathcal{F}$}
\psfrag{o}{$\mathcal{O}$}
\psfrag{h1}{$\cH^-$}
\psfrag{h2}{$\cH^+$}
\psfrag{b}{$\mathcal{B}$}\psfrag{beta}{$\beta$}\psfrag{pp}{$\cP^+$}
\includegraphics[width= 0.4\linewidth]{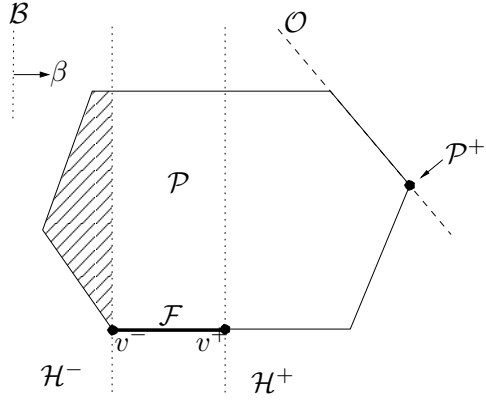}
\caption{Illustration for Theorem~\ref{Thm:NoIntersection}
\label{Fig:NS}}
\end{center}
\end{figure}

The previous discussion suggests that a first necessary condition for $\cP \overset{\cP}{\longrightarrow} \cF$
is that ${\cH^- \setminus \cF} \subset {\cP \setminus \cF}$ is empty. This is not quite right. There can be points
$x \in \cB_{v^-} \cap \cO$ which can still reach $\cF$. This is the content of the next lemma.

\begin{lemma}
\label{lem2}
Let $y, z$ be distinct points in $\mathcal{P}$ and let $l$ be the line segment joining them.
If $z, y \in \cO$ and $y \in \cB_z$, then $z \overset{l}{\longrightarrow} y$.
\end{lemma}
In light of this, we define the following set:
\[
\cB^- :=
\left\{\begin{array}{ll}
\cB_{v^-} \cap \cO
& \qquad \qquad \qquad \quad \text{if $\cB_{v^-} \cap \cO \cap \cF \neq \emptyset$} \\
\emptyset
& \qquad \qquad \qquad \quad \text{otherwise.}
\end{array}
\right.
\]
Lemma~\ref{lem2} says that points in $\cB^-$ can reach $\cF$, so these points should not be included in a candidate failure
set. Thus we arrive at our first necessary condition for $\cP \overset{\cP}{\longrightarrow} \cF$:
$\cH^- \setminus (\cF \cup \cB^-) = \emptyset$. Figure~\ref{Fig:NS} shows a shaded region corresponding to failure
of this condition.

Another failure leading to a second necessary condition is as follows. If $x_0 \in \cO$, then it
is on the boundary of $\cP$ and the instantaneous motion from
this point is only along $\cB$. If $\cB$ does not intersect the tangent cone of $\cP$ at $x_0$, then the only way to avoid
a trajectory leaving $\cP$ immediately is to place an equilibrium point at $x_0$.
The following lemma captures this situation by showing how $\cP$-invariant sets arise along $\cO$.
See the right side of Figure~\ref{Fig:NS}.

\begin{lemma}
\label{Lem:FaceInvariance}
Let $z$ be a point in $\mathcal{P}$.
If $\mathcal{B}_z \cap \mathcal{P} \subset \mathcal{O}$, then
$\mathcal{B}_z \cap \mathcal{P} $ and $\mathcal{P} \setminus \mathcal{B}_z$ are
$\mathcal{P}$-invariant.
\end{lemma}

A more subtle argument is needed to show that our proposed conditions are also sufficient to solve the reachability problem.
Sufficiency relies are two properties: the system is controllable, so it has sufficient maneuverability on $\cO$, and
the following lemma which provides the required maneuverability off of $\cO$.
\begin{lemma}
\label{Lem:Line}
Let $y \neq z \in \cP$ and let $l$ be the line segment joining them.
If $z, y \not\in \mathcal{O}$ and $y \in \ri{\cH}^-(z)$, then
$z \overset{l}{\longrightarrow} y$.
\end{lemma}

\begin{theorem}
\label{Thm:NoIntersection}
$\mathcal{P} \overset{\mathcal{P}}{\longrightarrow} \mathcal{F}$ ~~if and only if~~
(a) \ $\cH^- \setminus (\cF \cup \cB^-) = \emptyset$, and
(b) \ $\cP^+ \not\subset \cO \cap \ri{\cH}^+$.
\end{theorem}

\begin{proof}[Proof of Theorem~\ref{Thm:NoIntersection}.]
($\Longrightarrow$) First, suppose (a) does not hold. If
$\ri{\cH}^- \neq \emptyset$, then $\ri{\cH}^-$ is a
$\mathcal{P}$-invariant set by Lemma~\ref{lem1}, and it is in
$\mathcal{P} \setminus \mathcal{F}$, so the conclusion follows
from Proposition~\ref{Thm:NS}. Instead, if $\ri{\cH}^- =
\emptyset$, then $\cH^- = \cB_{v^-} \cap \mathcal{P}$. Since (a)
does not hold, there is $z \in (\mathcal{B}_{v^-} \cap
\mathcal{P}) \setminus (\cF \cup \cB^-)$. For this point $\beta^T
(Az+a) <0$, so for any $u$, $\beta^T (Az+a+Bu) <0$. This implies
any trajectory starting at $z$ immediately leaves $\cP$. Hence, $z
\not \overset{\cP}{ \longrightarrow} \cF$.

Second, suppose (b) does not hold, i.e. $\cP^+ \subset \cO \cap \ri{\cH}^+$.
Let $z \in \cP^+$ and notice that $\cP^+ =\cB_z \cap \cP$. Thus,
$\cB_z \cap \cP \subset \cO$. By
Lemma~\ref{Lem:FaceInvariance} it follows that $\cP^+$ is
$\mathcal{P}$-invariant. In addition, from
the assumption that $\cP^+ \subset \ri{\cH}^+$, $\cP^+ \subset \cP \setminus \cF$.
The conclusion follows from Proposition~\ref{Thm:NS}.

($\Longleftarrow$) Suppose conditions (a) and (b) hold.
For a point $x \in \cF$, it is trivial that $x \overset{\cP}{\longrightarrow} \cF$.

Let $x \in \cP \setminus (\cF \cup \cO)$.
By assumption (a) $x \notin {\cH}^-$ or equivalently $v^- \in \ri{\cH}^-(x)$.
Consequently, there is a point $y \in \cN(v^-) \cap \cF$ satisfying
$y \in \ri{\cH}^-(x)$, where $\cN(v^-)$ is a sufficiently
small neighborhood of $v^-$.  If $\mathcal{F}$ is not in
$\mathcal{O}$, such a point $y$ can be chosen not in
$\mathcal{O}$. Then these two points $x$ and $y$ satisfy the
assumption in Lemma~\ref{Lem:Line}, so $x
\overset{l}{\longrightarrow}  y$, where $l$ is the line segment
joining $x$ and $y$. Clearly, $l$ is in $\mathcal{P}$ as
$\mathcal{P}$ is convex. Hence, $x
\overset{\mathcal{P}}{\longrightarrow} \mathcal{F}$. Otherwise, suppose
$\mathcal{F} \subset \mathcal{O}$. It is easy to show that because $(A,B)$ is
controllable, $\mathcal{B} $ is not parallel to $\mathcal{O}$. It means we
can select a control $u$ so that $f(y,u)$ points outside of
$\mathcal{P}$. Thus, there is a sufficiently small $\epsilon > 0$
such that $\phi_t^u(y), t \in (-\epsilon, 0)$ is in
$\ri{\mathcal{P}} $. Note that $\phi_t^u(y)$ is continuous and $y
\in \ri{\cH}^-(x)$, so there is a point $z \in \phi_t^u(y),
t \in (-\epsilon, 0)$ satisfying $z \in \ri{\mathcal{H}}^-(x)$ and
therefore $z \in \ri{\mathcal{P}}$. Thus,  $\beta^T(Az+a) <0$ by
assumption. Applying Lemma~\ref{Lem:Line} for the two points
$x$ and $z$ leads to $x \overset{l}{\longrightarrow}  z$, where
$l$ is the line segment in $\mathcal{P}$ joining $x$ and $z$.
Considering $z \overset{\mathcal{P}}{\longrightarrow} y \in
\mathcal{F}$, we then have $x
\overset{\mathcal{P}}{\longrightarrow} \mathcal{F}$.

Finally, let $x \in (\cP \cap \cO) \setminus \cF$.
Clearly, $x$ is on the boundary of $\cP$. If $\cB_x \cap \mathcal{P}
\not \subset \mathcal{O}$, then  select a point $y \in
(\mathcal{B}_x \cap \mathcal{P} ) \setminus \mathcal{O}$ and let
$u$ be chosen such that $f(x,u)=Ax+a+Bu =y-x$, which is possible
because both $(Ax+a)$ and $(y-x)$ are in $\text{Im}(B)$. Note that
this vector field $f(x,u)$ points inside the polytope
$\mathcal{P}$. This implies the trajectory instantaneously enters
the interior of $\mathcal{P}$, which is not in $\mathcal{O}$ any
more. Then by the previous argument, it can be driven to reach
$\mathcal{F}$ through a line. Otherwise, if $\mathcal{B}_x \cap
\mathcal{P} \subset \mathcal{O}$, then the whole set
$\mathcal{B}_x \cap \mathcal{P}$ is on the boundary of
$\mathcal{P}$, and moreover it comprises either $\cP^+$ or
$\arg \min \{\beta^Tx: x \in \mathcal{F}\}$.
From condition (b), $\cP^+ \subset \cH^-(v^+)$ and this implies
$\cP^+ \cap \cF \neq \emptyset$. From condition (a),
$\ri{\cH}^- =\emptyset$ so
$\arg \min \{\beta^Tx: x \in \mathcal{F}\} \subset \mathcal{H}^+(v^-)$, which
implies $\arg \min \{\beta^Tx: x \in \mathcal{F}\} \cap \mathcal{F} \neq
\emptyset$. For both cases, we get $\mathcal{B}_x \cap \mathcal{F}
\neq \emptyset$. Now we select a point $y \in \mathcal{B}_x \cap
\mathcal{F}$. Then these two points satisfy the assumption in
Lemma~\ref{lem2}. Thus, it follows that $x \overset{l}{\longrightarrow} y$,
where $l$ is the line segment joining from $x$ to $y$, and so $x
\overset{\mathcal{P}}{\longrightarrow} \mathcal{F}$.
\end{proof}

Theorem~\ref{Thm:NoIntersection} gives necessary and sufficient
conditions for the reachability problem $\mathcal{P}
\overset{\mathcal{P}}{\longrightarrow} \mathcal{F} $. This result,
in turn, can be tied to failure sets, apropos
Proposition~\ref{Thm:NS}, which are the $\mathcal{P}$-invariant sets in
$\mathcal{P} \setminus \mathcal{F}$:
\begin{eqnarray}
\label{eq:FailureSets1} \hspace{-2.3cm} \cA^- & = & \cH^-
\setminus (\cF \cup \cB^-) \\ \label{eq:FailureSets2} \cA^+ & = &
\left\{\begin{array}{ll} \cP^+
& \qquad \qquad \qquad \quad \text{if $\cP^+ \subset \cO \cap \ri{\cH}^+$} \qquad \qquad \\
\emptyset
& \qquad \qquad \qquad \quad \text{otherwise.}
\end{array}
\right.
\end{eqnarray}

\begin{corollary}
\label{Cor:NonIntersect}
Let $\mathcal{A}=\mathcal{A}^- \cup \mathcal{A}^+$. Then
$\text{Reach}(\cP,\cF) = \cP \setminus \cA$. Moreover,
$\text{Reach}(\cP, \cF) \overset{\text{Reach}(\cP,\cF)}{\longrightarrow} \cF$ and
$\mathcal{A} \not \overset{\mathcal{P}}{ \longrightarrow} \mathcal{F}$.
\end{corollary}

We have identified $\text{Reach}(\cP,\cF)$, the maximal reachable set of $\mathcal{F}$ in
$\mathcal{P}$. This set, in
general, is not closed. This leads to difficulties with unbounded
control effort and unbounded time to reach $\mathcal{F}$.
Consequently, once failure sets have been identified, it is
desirable to remove them via a procedure that both
well-approximates the maximal reachable set and
also yields a closed $n$-dimensional polytope that can reach
$\mathcal{F}$. The approach is to ``cut off'' failure sets from
$\mathcal{P}$ by one of two procedures. One procedure is for
removing the failure $\mathcal{A}^-$ by cutting along
a hyperplane which is parallel to a slightly shifted version of
$\mathcal{B}$. The second procedure is for removing $\mathcal{A}^+$
by cutting exactly along a hyperplane parallel to
$\mathcal{B}$. These cuts are chosen arbitrarily close to the
failure sets and so that the remaining polytope has no failure
sets. It should be noted that the following procedure can be easily adapted to
convert explicit bounds on the controls to an appropriate $\epsilon$.

\begin{algorithm}
\label{Alg:PracticalCut} (Let $\epsilon>0$ be  sufficiently small.)

\begin{enumerate}
\item If $\mathcal{A}^- \neq \emptyset$, select affinely
independent points $z_1, \dots, z_k$ in $\mathcal{F} \cap
\mathcal{B}_{v^-}$ and also in the relative boundary of the facet
containing $\cF$, and then select points $z_{k+1}, \dots, z_n$
in $\ri{\mathcal{P}}\cap \ri{\cH}^+(v^-)$ such that
$\mathcal{G}:=\text{aff}\{z_1, \dots, z_n \}$ is of dimension
$n-1$ and $\max \limits_{x \in \mathcal{P} \cap \mathcal{G}}
\text{dist}(x, \mathcal{B}_{v^-}) =\epsilon$. Then divide
$\mathcal{P}$ along $\mathcal{G}$. \item If $\mathcal{A}^+ \neq
\emptyset$, select a point $z \in \mathcal{P}$ such that $\max
\limits _{x \in \mathcal{A}^+} \text{dist} (x, \mathcal{B}_z)
=\epsilon$. Then divide $\mathcal{P}$ along $\mathcal{B}_z$.
\end{enumerate}
\end{algorithm}

Let $\mathcal{A}_{\epsilon_-}$, $\mathcal{A}_{\epsilon_+}$,  and
$\text{Reach}_\epsilon(\cP,\cF)$ be the collection of sets after
the application of the division rules in
Algorithm~\ref{Alg:PracticalCut}, where $\mathcal{A}_{\epsilon_-}$
contains $\mathcal{A}^{-}$, $\mathcal{A}_{\epsilon_+}$ contains
$\mathcal{A}^+$, and $\text{Reach}_\epsilon(\cP,\cF)$ is the
remainder. Clearly, these three sets (if not empty) are
$n$-dimensional polytopes and $\mathcal{F} \subset
\text{Reach}_\epsilon(\cP,\cF) \subseteq
 \text{Reach}(\mathcal{P}, \mathcal{F})$. Then we have the
following corollary which follows directly from
Algorithm~\ref{Alg:PracticalCut} and
Theorem~\ref{Thm:NoIntersection}.

\begin{corollary}\label{Cor:PracticalCut}  \
\begin{enumerate}
\item[(a)] $\text{Reach}_\epsilon(\cP,\cF)
\overset{\text{Reach}_\epsilon(\cP,\cF)}{\longrightarrow}
\mathcal{F}$. \item[(b)] $\lim _{\epsilon \to 0}
\mu\left[\text{Reach}(\cP,\cF) \setminus
\text{Reach}_\epsilon(\cP,\cF) \right] =0$. \item[(c)] For any $x
\in \ri{\mathcal{P}}$, if $x
\overset{\mathcal{P}}{\longrightarrow} \mathcal{F}$ there exists
an $\epsilon>0$ such that $x \in \text{Reach}_\epsilon(\cP,\cF)$.
\end{enumerate}
\end{corollary}

\begin{example}
Consider the example in Figure~\ref{Fig:ce1} which illustrates  
the first step of Algorithm~\ref{Alg:PracticalCut}. 
\begin{figure}[!t]
\begin{center}
\psfrag{f}{$\cF$}\psfrag{beta}{$\beta$}
\includegraphics[width= 0.5 \linewidth]{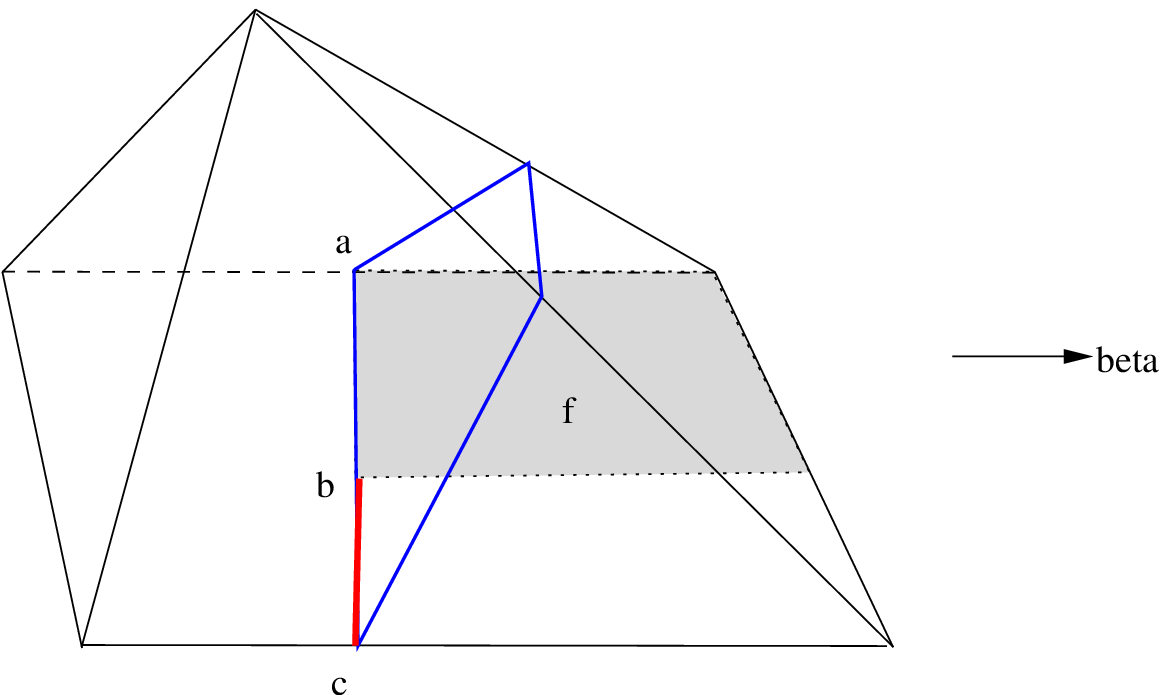}
\caption{ \label{Fig:ce1}}
\end{center}
\end{figure}
Suppose that there a failure set to reach $\cF$ along the red segment. If one were to cut
only along points in $\mathcal{F} \cap \mathcal{B}_{v^-}$ which corresponds to the blue
plane, then the failure set would not be cut off. Instead, if points in the relative boundary 
of the facet containing $\cF$ can be used, then this failure set can be removed. 
\end{example}
\begin{example}
\label{example1}
A simple example is presented to illustrate the possible failure sets and how
Algorithm~\ref{Alg:PracticalCut} cuts them off. Consider the system
\[
\begin{array}{l}
\dot x_1 =x_2, \\
\dot x_2 =u.
\end{array}
\]
It can be easily verified that $\mathcal{O}=\{(x_1, x_2): x_2=0\}$, the $x_1$ axis, and that $\mathcal{B}$ is
just the $x_2$ axis.
\begin{figure}[!t]
\begin{center}
\psfrag{v1}[][][\scale]{$v^-$} \psfrag{v2}[][][\scale]{$v^+$}
\psfrag{x1}[][][\scale]{$x_1$} \psfrag{x2}[][][\scale]{$x_2$}
\psfrag{f}[][][\scale]{$\mathcal{F}$}
\psfrag{o}[][][\scale]{$\mathcal{O}$}
\psfrag{a1}[][][\scale]{$\mathcal{A}^-$}
\psfrag{a2}[][][\scale]{$\mathcal{A}^+$}
\psfrag{p1}[][][\scale]{$\mathcal{A}_{\epsilon_-}$}
\psfrag{p2}[][][\scale]{$\mathcal{A}_{\epsilon_+}$}
\psfrag{b}[][][\scale]{$\mathcal{B}$}
\psfrag{beta}[][][\scale]{$\beta$}
\psfrag{pp}[][][\scale]{$\mathcal{P}'$}
\psfrag{ppp}[][][\scale]{$\text{Reach}_\epsilon(\cP,\cF)$}
\includegraphics[width= 0.99 \linewidth]{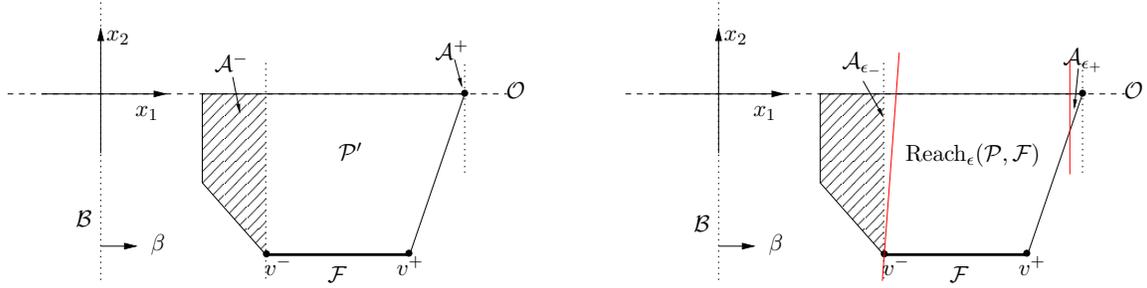}
\caption{An example. \label{Fig:Ex1}}
\end{center}
\end{figure}
Suppose that the polytope $\mathcal{P}$ and the target set
$\mathcal{F}$ are as shown in Figure~\ref{Fig:Ex1}. The hyperplane
$\mathcal{O}$ touches the polytope but has empty intersection with
its interior. We get $\mathcal{A}^-=\cH^- \setminus \cF$ which is
the patterned region in Figure~\ref{Fig:Ex1}; and $\mathcal{A}^+$
is just a point. $\text{Reach}(\mathcal{P}, \mathcal{F})$ is the
set $\mathcal{P}'$ not including the boundary of $\mathcal{A}^-$
and $\mathcal{A}^+$. This set is not closed. Moreover, if an
initial state $x_0 \in \mathcal{P}'$ approaches the boundary of
$\mathcal{A}^-$, the control input $u(x_0)$ tends to infinity in
order to reach $\mathcal{F}$ with constraint in $\mathcal{P}$.
Also, if $x_0 \in \mathcal{P}'$ approaches the boundary of
$\mathcal{A}^+$, the time to reach $\mathcal{F}$ tends to
infinity. Applying Algorithm~\ref{Alg:PracticalCut}, a good closed
$\epsilon$-approximation $ \text{Reach}_\epsilon(\cP,\cF)$ of
$\mathcal{P}'$ is given on the right of Figure~\ref{Fig:Ex1}.
\end{example}

\section{Control Synthesis on Simplices}
\label{sec:simplices}

Consider an $n$-dimensional simplex $\mathcal{S}$ with vertices
$v_0, v_1, \dots, v_n$ and facets $\mathcal{F}_0, \dots,
\mathcal{F}_n$ (the facet is indexed by the vertex not contained).
Let $I := \{ 1,\ldots,n \}$.
\begin{problem}\label{Prbm:CTF}
Consider system (\ref{eq:AffineSystem}) defined on $\mathcal{S}$.
Find an affine feedback control $u=Fx+g$ such  that for
every $x_0 \in \mathcal{S}$ there exist $T \geq 0 $ and $\epsilon
>0$ satisfying:
\begin{enumerate}
\item[(i)] $\phi_t^u(x_0) \in \mathcal{S}$ for all $t \in [0, T]$;
\item[(ii)] $\phi_T^u(x_0) \in \mathcal{F}_0$;
\item[(iii)] $\phi_t^u(x_0) \notin \mathcal{S}$ for all $t \in (T, T+\epsilon)$.
\end{enumerate}
\end{problem}
Condition (iii) is interpreted to mean that the closed-loop
dynamics on $\mathcal{S}$ are extended to a neighborhood of
$\mathcal{S}$.

\begin{definition}
The {\em invariance conditions} for $\cS$ require that there exist
$u_0,\ldots,u_n \in \RR^m$ such that:
\begin{equation}
\label{eq:inv}
h_j \cdot (Av_i+a+Bu_i) \leq 0 \,, \qquad \qquad \qquad i \in \{0,\dots, n\}, \qquad j \in I \setminus \{ i \} \,.
\end{equation}
\end{definition}

\begin{theorem}{\cite{HVS06,RB06}}
\label{Thm:LinearSimplex}
Given the system \eqref{eq:AffineSystem} and an affine feedback $u(x)=Kx+g$, 
where $K \in \RR^{m \times n}$, $g \in \RR^m$, and $u_0 = u(v_0),\ldots,u_n=u(v_n)$, 
the closed-loop system satisfies $\cS \overset{\cS}{\longrightarrow} \cF_0$ if and only if 
\begin{enumerate}
\item[(a)] 
The invariance conditions \eqref{eq:inv} hold. 
\item[(b)] 
There is no equilibrium in $\mathcal{\cS}$.
\end{enumerate}
\end{theorem}

Theorem~\ref{Thm:LinearSimplex} cannot be used directly for our present work because it enforces that affine feedbacks
be used. Unfortunately, this class is not large enough if solvability of RCP by open-loop control is the starting 
point. The next result shows that for hypersurface systems on simplices, one sufficiently rich 
feedback class is piecewise affine feedbacks. The proof is in the Appendix.
\begin{theorem}{\cite{LB06}}
\label{Prop:ControlforSimplex}
If $\mathcal{S} \overset{\mathcal{S}}{\longrightarrow} \mathcal{F}_0$
then there exists a piecewise affine state feedback
$u=F_{\sigma} x + g_{\sigma}$, $\sigma \in \{1, 2\}$
that also achieves $\mathcal{S} \overset{\mathcal{S}}{\longrightarrow} \mathcal{F}_0$.
\end{theorem}

\section{Control Synthesis on Polytopes}
\label{sec:controlpoly}

We now begin our investigation of state feedback synthesis on
polytopes. We want to show that if $\mathcal{P}
\overset{\mathcal{P}}{\longrightarrow} \mathcal{F}$ using
open-loop control then there exists a piecewise affine
feedback solving the reachability problem. The idea is to
triangulate the polytope, transform the reachability problem
within a polytope into a set of reachability problems for
simplices, and then devise appropriate piecewise affine
controllers on each simplex using Proposition~\ref{Prop:ControlforSimplex} of
the previous section. The triangulation must be performed properly
otherwise the procedure may fail. First we present a lemma that aids
in finding a  proper triangulation.
\begin{lemma}
\label{Lem:ExtremePoint}
If $\mathcal{P} \overset{\mathcal{P}}{\longrightarrow} \mathcal{F}$, then there exists a vertex
$v_*$ of $\mathcal{P}$ in $\cP^+$ such that either $v_* \notin \mathcal{O}$ or $v_* \in \mathcal{F}$.
\end{lemma}
\begin{proof}
Suppose by contradiction that for any  vertex  $v \in \cP^+$ we have
$v \in \cO$ and $v \notin \cF$. Note that $v \notin \mathcal{F}$ for all
vertices $v \in \cP^+$ implies, by convexity,
$\cP^+ \subset \ri{\cH}^+$. Moreover, since $v \in \mathcal{O}$ for all $v \in \cP^+$,
it follows from the convexity of $\mathcal{O}$
that $\cP^+ \subset \cO$.  Hence, by Theorem~\ref{Thm:NoIntersection}, this contradicts $\mathcal{P}
\overset{\mathcal{P}}{\longrightarrow} \mathcal{F}$.
\end{proof}

We review some concepts on triangulation \cite{JacJos97}.
Suppose $\mathcal{V}$ is a finite set of points such that $\text{conv}(\mathcal{V})$ is $n$-dimensional.
A {\em subdivision} of $\mathcal{V}$ is a finite collection $\PP=\{\mathcal{P}_1, \dots, \mathcal{P}_m\}$ of
$n$-dimensional polytopes such that the vertices of each $\mathcal{P}_i$ are drawn from $\mathcal{V}$;
$\text{conv}(\mathcal{V})$ is the union of
$\mathcal{P}_1, \dots, \mathcal{P}_m$; and $\mathcal{P}_i \cap \mathcal{P}_j$ ($ i\neq j$)
is a common (possibly empty) face of  $\mathcal{P}_i$ and $\mathcal{P}_j$.
A {\em triangulation} of $\mathcal{V}$ is a subdivision in which each $\mathcal{P}_i$ is a simplex.
In the following we assume that $\cF$ is a facet of $\cP$.

{\bf Basic Triangulation of $\cP$:} \label{triangle1} \
\begin{enumerate}
\item
Select $v_* $ as in Lemma~\ref{Lem:ExtremePoint}.
\item
Triangulate each facet $\mathcal{F}_j$ of $\mathcal{P}$.
Denote $\{\mathcal{S}^i_{\mathcal{F}_j}: i =1, \dots, k_j\}$ the triangulation for $\mathcal{F}_j$.
\item
Let $\mathbb{S} =\{\mathcal{S}_1, \dots, \mathcal{S}_q\} := \{\text{conv}(v_*, \mathcal{S}^i_{\mathcal{F}_j}) : \mathcal{F}_j \text{ is any facet of $\mathcal{P}$ not containing } v_* \}$.
\end{enumerate}

\begin{lemma}
\label{Lem:Triangulation}
The collection $\mathbb{S}$ is a triangulation of $\text{vert}(\mathcal{P}) \cup \text{vert}(\mathcal{F})$ such
that every simplex in $\mathbb{S}$ contains $v_*$ as a vertex.
\end{lemma}
\proof
By construction, it is clear that every simplex $\mathcal{S}_i \in \mathbb{S}$ contains $v_*$ as a vertex, the vertices of $\mathcal{S}_i$ are drawn from $\text{vert}(\mathcal{P}) \cup \text{vert}(\mathcal{F})$, and $\mathcal{S}_i \cap \mathcal{S}_j$ ($i\neq j$) is a common (possibly empty) face of $\mathcal{S}_i$ and $\mathcal{S}_j$. Next, we show that $\mathcal{P}$ is the union of $\mathcal{S}_1, \dots, \mathcal{S}_q$. Let $x$ be a point in the union of $\mathcal{S}_1, \dots, \mathcal{S}_q$. Then it must be in a simplex $\mathcal{S}_i$. Thus, by convexity of $\mathcal{P}$, $x \in \mathcal{P}$. On the other hand, let $x$ be a point in $\mathcal{P}$. Draw a line through $v_*$ and $x$. It intersects at a point $y$ with a facet (say $\mathcal{F}_j$) of $\mathcal{P}$ that does not contain $v_*$. It means there exists
a simplex $\mathcal{S}^i_{\mathcal{F}_j}$ containing $y$. So $x \in \text{conv}(v_*, \mathcal{S}^i_{\mathcal{F}_j})$, one of the
simplices in $\mathbb{S}$. The conclusion follows.
\endproof

Now suppose we have a triangulation $\mathbb{S}=\{\mathcal{S}_1 ,
\dots, \mathcal{S}_q\}$ as above, and denote $\cS_0 :=
\mathcal{F}$. We say $\mathcal{S}_i$ and $ \mathcal{S}_j$ are {\em
adjacent} (denoted by $\mathcal{S}_i \sim \mathcal{S}_j$) if
$\cF_{ij} := \mathcal{S}_i \cap \mathcal{S}_j$ is a facet. A
sequence $(\mathcal{S}_{i_k}, \dots, \mathcal{S}_{i_0})$ is called
a {\em path} to reach $\mathcal{S}_{i_0}$ if $\mathcal{S}_{i_j}
\sim \mathcal{S}_{i_{j-1}}$ and $\mathcal{S}_{i_j}
\overset{\mathcal{S}_{i_j}}{\longrightarrow}
\mathcal{S}_{i_{j-1}}$ for each $1 \leq j\leq k$. The {\em length}
of such a path is $k$. We propose a greedy algorithm that orders
simplices according to minimum $\beta$ component of exit vertices
first. More precisely, at every iteration a pair
$(\mathcal{S}_i, \mathcal{S}_j)$ is selected that minimizes the
$\beta$-component of any vertex on the exit facet $\cF_{ij}$. If
there is more than one pair achieving the minimum, select a pair
which has the maximum number of exit vertices achieving the
minimum. In the algorithm below $\mathcal{R}_f$ and
$\mathcal{R}_u$ denote the finished and unfinished set of
simplices, respectively, and let
\[
w' \in \arg \min \biggl\{\beta ^T x ~:~ x \in 
\bigcup \left\{ \cS_k \cap \cS_l ~:~ (\cS_k, \cS_l) \in \cR_u \times \cR_f \text{ satisfying } \cS_k
\sim \cS_l \right\} \biggr\} \,.
\]

{\bf Greedy algorithm for path generation in $\mathbb{S}$:}
\begin{enumerate}
\item Initialization: $\mathcal{R}_{f}:=\{\mathcal{S}_0\}$,
$\mathcal{R}_{u}:=\{\mathcal{S}_1, \dots, \mathcal{S}_q\}$; \item
While $(\cR_{u} \neq \emptyset)$, choose $(\cS_i, \cS_j) \in \cR_u
\times \cR_f$ such that $\cS_i  \sim \cS_j$, it achieves   $\min
\limits_{x \in \cF_{ij}} \beta ^T x  =\beta ^T w'$, and
$\mathcal{F}_{ij}$ contains the maximum number of vertices in
$\mathcal{B}_{w'}$. Then move $\cS_i$ from $\cR_u$ to $\cR_f$.
\end{enumerate}
Once the greedy algorithm has generated paths, the synthesis of a piecewise affine control is straightforward.
See also \cite{HVS06}.

{\bf Piecewise affine synthesis:}
\begin{enumerate}
\item
Let $\{\dots,(\dots, \mathcal{S}_i, \mathcal{S}_j, \dots, \mathcal{S}_0),  \dots\}$ be a collection of paths
to reach $\mathcal{S}_0$;
\item
Find $u^i(x):=F_{\sigma_i(x)}x+g_{\sigma_i(x)}, \ i=1, \dots, q$,
that solves $\mathcal{S}_i \overset{\mathcal{S}_i}{\longrightarrow} \mathcal{F}_{ij}$,
where $\mathcal{F}_{ij}$ is the common facet of $\mathcal{S}_i$ and the next simplex in the path;
\item
For all $x \in \mathcal{S}_i \in \mathbb{S}$, let $u(x)=u^i(x)$.
If $x \in \mathcal{P}$ belongs to more than one simplex, set $u(x)=u^j(x)$ where $j$ is the index of
a simplex that has the shortest path to reach $\mathcal{S}_0$.
\end{enumerate}

\begin{theorem}
\label{Thm:ControlforPolytope}
Suppose that $\cF$ is a facet of $\cP$.
There exists a piecewise affine state feedback that achieves
$\mathcal{P} \overset{\mathcal{P}}{\longrightarrow} \mathcal{F}$
if and only if $\mathcal{P} \overset{\mathcal{P}}{\longrightarrow} \mathcal{F}$ using open-loop control.
\end{theorem}

The idea of the proof is to show that the path generation algorithm does not terminate until $\cR_u = \emptyset$
by showing that for the next selected pair $(\mathcal{S}_i, \mathcal{S}_j) \in \mathcal{R}_{u} \times \mathcal{R}_{f}$,
the reachability problem $\mathcal{S}_i  \overset{\mathcal{S}_i}{\longrightarrow} \mathcal{S}_j$ can be solved.
This is done by applying Theorem~\ref{Thm:NoIntersection} and verifying conditions (a) and (b) for the selected pair
$(\mathcal{S}_i, \mathcal{S}_j) \in \mathcal{R}_{u} \times \mathcal{R}_{f}$. The main effect of our selection of triangulation
based on vertex $v_\ast$ is that condition (b) holds trivially for any such pair. The fact that condition (a) can be made to
hold is the main feature of the greedy strategy with respect to $\beta$. This strategy
guarantees that the vertex $v_0 \in \cS_i$ not contained in the exit facet has a strictly larger
$\beta$-component, and this means that failure set $\cA^- = \emptyset$ for $\cS_i$. The proof now easily follows from these
observations.

\proof
($\Longrightarrow$) Obvious. \
($\Longleftarrow$)
If the path generation algorithm terminates with $\mathcal{R}_{u} = \emptyset$ then by straightforward dynamic programming
arguments there exists a piecewise affine feedback control that achieves
$\mathcal{P} \overset{\mathcal{P}}{\longrightarrow} \mathcal{F}$. It is therefore sufficient to show that if
$\mathcal{R}_{u} \neq \emptyset$, there exists a pair
$(\mathcal{S}_i, \mathcal{S}_j) \in \mathcal{R}_{u} \times \mathcal{R}_{f}$ such that
$\mathcal{S}_i \cap \mathcal{S}_j =: \mathcal{F}_{ij}$ is a facet and
$\mathcal{S}_i  \overset{\mathcal{S}_i}{\longrightarrow} \mathcal{S}_j$.

Consider any pair $(\cS_i, \cS_j) \in \mathcal{R}_{u} \times \mathcal{R}_{f}$. We must verify conditions (a)
and (b) of Theorem~\ref{Thm:NoIntersection} to show $\mathcal{S}_i  \overset{\mathcal{S}_i}{\longrightarrow} \mathcal{F}_{ij}$.
Consider condition (b). We have two observations about $v_\ast$. First,
from Lemma~\ref{Lem:Triangulation}, $v_\ast \in \mathcal{S}_i$, $\forall i$, and
therefore $v_\ast \in \mathcal{F}_{ij}$. Second,
$v_\ast \in \cP^+$ implies $v_\ast \in \cS_i^+$.
Applying these two facts, condition (b) for
$\mathcal{S}_i  \overset{\mathcal{S}_i}{\longrightarrow} \mathcal{S}_j$ says that
$\cS_i^+ \not\subset \cO \cap \{ x \in \cS_i: \beta^T x > \beta^T v_\ast \}$, and this is obviously true.

So far we have shown that for any pair $(\mathcal{S}_i,
\mathcal{S}_j) \in \mathcal{R}_{u} \times \mathcal{R}_{f}$ as
above, condition (b) of Theorem~\ref{Thm:NoIntersection} holds for
the problem $\mathcal{S}_i
\overset{\mathcal{S}_i}{\longrightarrow} \mathcal{F}_{ij}$. Now we
will show that for the selected pair $(\mathcal{S}_i,
\mathcal{S}_j)$, condition (a) holds. Let $v_0$ be the vertex of
$\mathcal{S}_i$ not in $\mathcal{F}_{ij}$. Let $w \in \cF_{ij}
\cap  \cB_{w'}$. There are three cases.
\begin{enumerate}
\item Suppose $\beta^T w < \beta^T v_0$. Then condition (a) holds.
\item Suppose $\beta^T w > \beta^T v_0$. Also, we know $\beta^T
v^- \le \beta^T v_0 < \beta^T w$ from the assumption $\cP
\overset{\cP}{\longrightarrow} \cF$. By convexity, for every point
$y$ on the line segment joining $v^-$ and $v_0$, $\beta^T y <
\beta^T w$. However, $v^- \in \mathcal{S}_0 \in \mathcal{R}_{f}$
and $v_0 \in \mathcal{S}_i \in \mathcal{R}_u$, which means the
line segment contains a point $y$ on the boundary of
$\mathcal{S}_{i'} \in \mathcal{R}_u$ and $\mathcal{S}_{j'} \in
\mathcal{R}_f$. This contradicts the choice of the pair
$(\mathcal{S}_{i}, \mathcal{S}_{j})$ that achieves $\min
\limits_{x \in \cF_{ij}} \beta ^T x  =\beta ^T w'$. \item Suppose
$\beta^T w = \beta^T v_0$. Let $\{ v_1,\ldots,v_k \}$ be the set
of vertices of $\mathcal{F}_{ij}$ that lie in $\mathcal{B}_{w'}$.
If $\cB_{w'} \cap \cP \subset \mathcal{O}$ then condition (a)
holds and we are done. If not, it follows from the assumption
$\mathcal{P} \overset{\mathcal{P}}{\longrightarrow} \mathcal{F}$
that either $\cB_{w'} \cap \cP \subset \mathcal{F}$ or $\beta^T w'
> \beta^T v^-$. For both cases we claim that
$\mathcal{G}:=\text{conv}\{ v_0, v_1, \ldots, v_k \}$ belongs to
some $\mathcal{S}_k \in \mathcal{R}_f$. For the former case, it is
obvious since $\mathcal{G} \subset \cB_{w'} \cap \cP \subset \cS_0
\in \mathcal{R}_f$. For the latter case, suppose not. Say a point
$x \in \mathcal{G}$ does not belong to some $\mathcal{S}_k \in
\mathcal{R}_f$. Then since the union of sets in $\mathcal{R}_f$ is
a closed set, there exists a point $y \in \mathcal{P}$ near $x$
satisfying $\beta^T y < \beta^T w'$, and $y$ also does not belong
to some $\mathcal{S}_k \in \mathcal{R}_f$. This contradicts the
choice of the pair $(\mathcal{S}_{i}, \mathcal{S}_{j})$ that
achieves $\min \limits_{x \in \cF_{ij}} \beta ^T x  =\beta ^T w'$.
Therefore $\cG$ belongs to some $\cS_k \in \cR_f$ which implies it
belongs to some facet $\mathcal{F}_{i'j'} \neq \mathcal{F}_{ij}$
with $\mathcal{F}_{i'j'} = \mathcal{S}_{i'} \cap
\mathcal{S}_{j'}$, where $\mathcal{S}_{i'} \in \mathcal{R}_u$,
$\mathcal{S}_{j'} \in \mathcal{R}_f$, and $\mathcal{F}_{i'j'}$ has
one more vertex, namely $v_0$, in $\cB_{w'}$. This contradicts the
choice of $\mathcal{F}_{ij}$.
\end{enumerate}
\endproof

\begin{example}
Consider again Example~\ref{example1}. After applying
Algorithm~\ref{Alg:PracticalCut} to cut the failure sets off, we
know from Corollary~\ref{Cor:PracticalCut} that
$\text{Reach}_\epsilon(\cP,\cF)
\overset{\text{Reach}_\epsilon(\cP,\cF)} {\longrightarrow}
\mathcal{F}$ using open-loop control. We want to find a piecewise
affine state feedback that achieves
$\text{Reach}_\epsilon(\cP,\cF)
\overset{\text{Reach}_\epsilon(\cP,\cF)} {\longrightarrow}
\mathcal{F}$.
\begin{figure}[!t]
\begin{center}
\psfrag{v1}{$v_1$} \psfrag{v2}{$v_2$} \psfrag{v3}{$v_3$}
\psfrag{s1}{$\mathcal{S}_1$} \psfrag{s2}{$\mathcal{S}_2$}
\psfrag{s3}{$\mathcal{S}_3$} \psfrag{v4}{$v_4(v_-)$}
\psfrag{v5}{$v_5(v_+)$} \psfrag{x1}{$x_1$} \psfrag{x2}{$x_2$}
\psfrag{f}{$\mathcal{F}$} \psfrag{o}{$\mathcal{O}$}
\psfrag{b}{$\mathcal{B}$} \psfrag{beta}{$\beta$}
\psfrag{ppp}{$\text{Reach}_\epsilon(\cP,\cF)$}
\includegraphics[width= 0.6 \linewidth]{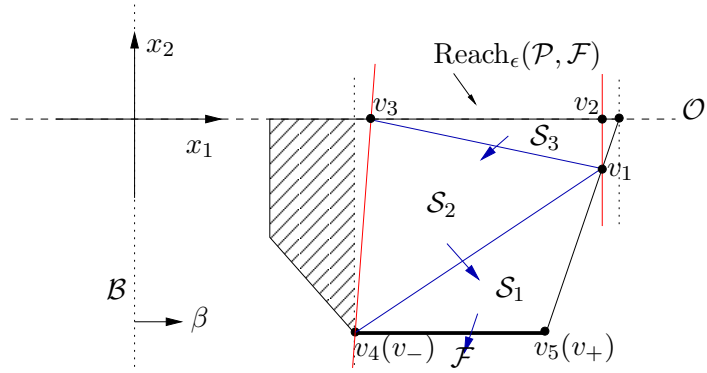}
\caption{A triangulation of $\text{Reach}_\epsilon(\cP,\cF)$ and a
path to reach $\mathcal{F}$. \label{Fig:Ex2}}
\end{center}
\end{figure}
Denote the vertices of $\text{Reach}_\epsilon(\cP,\cF)$ by $v_1,
\dots, v_5$ in Figure~\ref{Fig:Ex2}. It can be easily obtained
that $v_1$ is the only vertex of $\text{Reach}_\epsilon(\cP,\cF)$
satisfying the property of Lemma~\ref{Lem:ExtremePoint} so
$v_*=v_1$. Also, $v_* \not \in \mathcal{F}$. By the proposed
triangulation method, we obtain a triangulation
$\mathcal{S}=\{\mathcal{S}_1, \mathcal{S}_2, \mathcal{S}_3\}$ as
shown in Figure~\ref{Fig:Ex2}. By applying
Theorem~\ref{Thm:NoIntersection} to each simplex, it can be easily
checked that $\mathcal{S}_3
\overset{\mathcal{S}_3}{\longrightarrow} \mathcal{S}_2$,
$\mathcal{S}_2 \overset{\mathcal{S}_2}{\longrightarrow}
\mathcal{S}_1$, and $\mathcal{S}_1
\overset{\mathcal{S}_1}{\longrightarrow} \mathcal{F}$. Thus, we
can find a control to solve the reachability problem on each
simplex (based on Proposition~\ref{Prop:ControlforSimplex}) and
then we can construct a piecewise affine control which achieves
$\text{Reach}_\epsilon(\cP,\cF)
\overset{\text{Reach}_\epsilon(\cP,\cF)} {\longrightarrow}
\mathcal{F}$.
\end{example}

\section{Triangulation with respect to $\cF$}
\label{sec:trianglewrtF}

In this section we study how the previous results can be extended to solve the control synthesis problem if $\cF$
is not given as a facet of $\cP$. If
the designer has flexibility in modifying the given state constraints, then one
perform a slight modification (by pulling out $\cF$) so that $\cF$ is a facet of a larger
polytope $\cP'$. However, this approach has two caveats:
(1) The problem $\cP' \overset{\cP'}{\longrightarrow} \cF'$ may not be solvable even if
$\cP \overset{\cP}{\longrightarrow} \cF$ is; (2) If $\cP$ is part of a larger subdivision of the state space, then
possibly other polytopes in the subdivision must be modified.
A more desirable procedure is to use a triangulation method that refines the given subdivision of the
state space by splitting $\cP$ so that $\cF$ becomes a facet of one of the polytopes in the refined subdivision.
This approach also has pitfalls, because if one does not refine the subdivision properly, failure sets may emerge even if
$\cP \overset{\cP}{\longrightarrow} \cF$ by open-loop control. In this section we show one method (among several)
to obtain a proper triangulation.

Let $\bar{\cF}$ denote the facet of $\cP$ containing $\cF$. First we consider a simple case when $v_{\ast}$ of
Lemma~\ref{Lem:Triangulation} can be selected so that $v_{\ast} \not\in \bar{\cF}$. See Figure~\ref{Fig:Ill1}.
\begin{figure}[!t]
\begin{center}
\psfrag{v5s}{$v_5(v_*)$}
\psfrag{v3}{$v_3$}
\psfrag{v6}{$v_6$}
\psfrag{v5}{$v_5$}
\psfrag{v2}{$v_2$}
\psfrag{v4}{$v_4$}
\psfrag{v1}{$v_1$}
\psfrag{vs}{$v_3(v_*)$}
\psfrag{f}{$\mathcal{F}$}
\psfrag{beta}{$\beta$}
\includegraphics[width= 0.5 \linewidth]{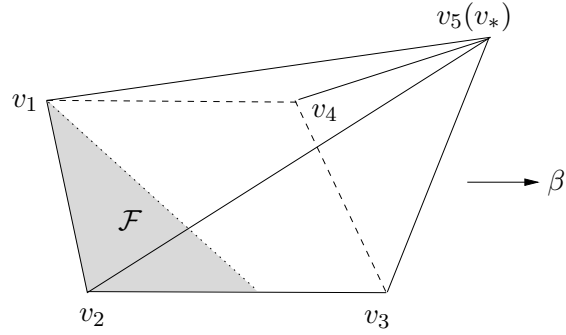}
\caption{$\mathcal{F}$ is not a facet of $\cP$ but $v_* \not \in \bar{\mathcal{F}}$. \label{Fig:Ill1}}
\end{center}
\end{figure}

{\bf Triangulation of $\cP$ with respect to $\cF$:}
\label{triangulation2} \
\begin{enumerate}
\item[(a)]
Select $v_*$ as in Lemma~\ref{Lem:ExtremePoint} and so that $v_\ast \not\in \bar{\cF}$.
\item[(b)]
Make a triangulation of $\text{vert}(\bar{\cF}) \cup \text{vert}(\mathcal{F})$ such that the interior of each
resulting simplex is either entirely in $\mathcal{F}$ or not in $\mathcal{F}$. For the remaining facets
$\mathcal{F}_j$ of $\mathcal{P}$, make a triangulation of $\text{vert}(\mathcal{F}_j)$.
Denote $\{\mathcal{S}^i_ {\mathcal{F}_j}: i =1, \dots, k_j\}$ the triangulation for $\mathcal{F}_j$.
\item[(c)]
Let $\mathbb{S} =\{\mathcal{S}_1, \dots, \mathcal{S}_q\} := \{\text{conv}(v_*, \mathcal{S}^i_{\mathcal{F}_j}) : \mathcal{F}_j \text{ is any facet of $\mathcal{P}$ not containing } v_* \}$.
\end{enumerate}

The first thing we notice is that nothing about the proof of Lemma~\ref{Lem:Triangulation} is specific to $\cF$ being
a facet, so the lemma still holds for the new triangulation. Also the proof of Theorem~\ref{Thm:ControlforPolytope}
is unchanged since the essential property of $v_{\ast}$ (namely Lemma~\ref{Lem:Triangulation})
is still true. Therefore, we have the following direct extension of Theorem~\ref{Thm:ControlforPolytope}.
\begin{corollary}
\label{Cor:FnotFacet1}
Suppose that $\cF$ is not a facet of $\cP$ and there exists $v_{\ast}$ as in Lemma~\ref{Lem:ExtremePoint} such
that $v_{\ast} \not\in \bar{\cF}$.
There exists a piecewise affine state feedback that achieves
$\mathcal{P} \overset{\mathcal{P}}{\longrightarrow} \mathcal{F}$
if and only if $\cP \overset{\cP}{\longrightarrow} \cF$ using open-loop control.
\end{corollary}

When there does not exist $v_\ast \not\in \bar{\cF}$, the problem is more complex because
Lemma~\ref{Lem:Triangulation} breaks down. Nevertheless, we would like to build upon our previous
triangulation and control methods by appropriately subdividing $\cP$. A natural idea would be to form
$\cP_1 := \conv(\cF, v ~|~ v \in V \setminus \bar{\cF})$,
a polytope for which $\cF$ is a facet. There are two problems to
be addressed. First, can $\cP_1$ have failure sets for the problem
$\cP_1 \overset{\cP_1}{\longrightarrow} \cF$ even if $\cP \overset{\cP}{\longrightarrow} \cF$?
Theorem~\ref{Thm:NoIntersection} tell us that $\cH^- \setminus (\cF \cap \cB^-) = \emptyset$ and we
observe that this condition is identical for any polytope with the same exit facet $\cF$. Therefore,
condition (a) holds for $\cP_1 \overset{\cP_1}{\longrightarrow} \cF$. Instead, it is condition (b)
which is problematic because generally $\cP^+_1 \neq \cP^+$ and equilibria can appear on $\cP_1^+$ when
we try to solve $\cP_1 \overset{\cP_1}{\longrightarrow} \cF$. A more careful approach is needed,
and inspiration is provided by the proof of Theorem~\ref{Thm:ControlforPolytope}: for any $n$-dimensional
polytope $\tilde{P} \subset \cP$ with exit facet $\cF$, if $\tilde{\cP}^+ \cap \cF \neq \emptyset$, then
condition (b) automatically holds. See Figure~\ref{Fig:Ill3} for an example. Thus, we have the following.
\begin{figure}[!t]
\begin{center}
\psfrag{v6}{$v_6$}
\psfrag{v5}{$v_5$}
\psfrag{v2}{$v_2$}
\psfrag{v4}{$v_4$}
\psfrag{vm}{$v_1(v_-)$}
\psfrag{vs}{$v_3(v_*) \in \mathcal{O}$}
\psfrag{f}{$\mathcal{F}$}
\psfrag{beta}{$\beta$}
\includegraphics[width= 0.6 \linewidth]{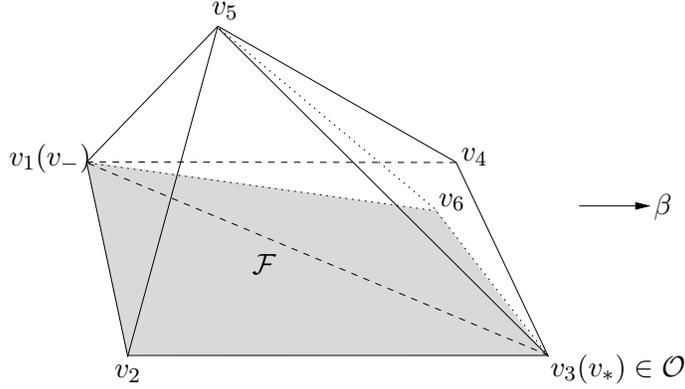}
\caption{$\mathcal{F}=\text{conv}\{v_1, v_2, v_3, v_6\} \subset \bar{\mathcal{F}}$
and $v_*=v_3 \in \mathcal{F}$. \label{Fig:Ill3}}
\end{center}
\end{figure}

\begin{proposition}
\label{prop1}
Suppose there exists $v_{\ast}$ a vertex of $\cF$ such that $v_{\ast} \in \cF \cap \cP^+$.
Let $\tilde{\cP} \subset \cP$ be an $n$-dimensional polytope
such that $\cF$ is a facet of $\tilde{\cP}$. Then $\cP \overset{\cP}{\longrightarrow} \cF$ implies
$\tilde{\cP} \overset{\tilde{\cP}}{\longrightarrow} \cF$.
\end{proposition}

\begin{proof}
Consider condition (b) of Theorem~\ref{Thm:NoIntersection}
for $\tilde{\cP} \overset{\tilde{\cP}}{\longrightarrow} \cF$. We have to show that
$\tilde{P}^+ \not\subset \cO \cap \{ x \in \tilde{P} ~|~ \beta^T x > \beta^T v^+ \}$. But $v_{\ast} \in \cF \cap \cP^+$
implies $\tilde{P}^+ = \{ x \in \tilde{P} ~|~ \beta^T x = \beta^T v^+ \}$, so condition (b) is obviously true.

For condition (a), Theorem~\ref{Thm:NoIntersection} tells us that
$\cH^- \setminus (\cB^- \cup \cF) = \emptyset$ and since $\tilde{H}^- = \cH^-$ and
$\tilde{B}^- = \cB^-$, condition (a) obviously holds for $\tilde{\cP} \overset{\tilde{\cP}}{\longrightarrow} \cF$.
\end{proof}

Proposition~\ref{prop1} gives some indication of how the polytope $\cP_1$ which has $\cF$ as a facet could be constructed.
Now we face the second problem. The set $\cP \setminus \cP_1$ is of course not a polytope. How shall it be subdivided
and what reachability problems need to be assigned to avoid new failure sets from appearing? The problem is
difficult due to the generality of the description of $\cF$. However, the following proposition gives some indication
of how other polytopes can be constructed which do not have $\cF$ as their exit facet.

\begin{proposition}
\label{prop2} Suppose there exists $v_{\ast}$, a vertex of $\cF$,
such that $v_{\ast} \in \cF \cap \cP^+$. Let $\tilde{\cP} \subset
\cP$ be an $n$-dimensional polytope and let $\tilde{\cF}$ be an
$(n-1)$-dimensional polytope which is a facet of $\tilde{\cP}$.
Suppose there exist $\tilde{v}^- \in \tilde{\cF} \cap \arg\min\{
\beta^T x ~|~ x \in \cP \}$ and $\tilde{v}^+ \in \tilde{\cF} \cap
\cP^+$. Then $\cP \overset{\cP}{\longrightarrow} \cF$ implies
$\tilde{\cP} \overset{\tilde{\cP}}{\longrightarrow} \tilde{\cF} $.
\end{proposition}

\begin{proof}
By the assumption $\tilde{v}^+ \in \tilde{\cF} \cap \cP^+$ and by
the same argument as in Proposition~\ref{prop1}, condition (b) for
$\tilde{\cP} \overset{\tilde{\cP}}{\longrightarrow} \tilde{\cF}$
obviously holds. Consider condition (a) for $\cP
\overset{\cP}{\longrightarrow} \cF$. It says that $\{ x \in \cP
~|~ \beta^T x \le \beta^T v^- \} \setminus (\cF \cup \cB^-) =
\emptyset$. Equivalently, $\{ x \in \cP ~|~ \beta^T x < \beta^T
v^- \} = \emptyset$ and $\cB_{v^-} \cap \cP \subset \cO \cup \cF$.
Because $\tilde{v}^- \in \tilde{\cF} \cap \arg\min\{ \beta^T x ~|~
x \in \cP \}$, this means $\{ x \in \tilde{\cP} ~|~ \beta^T x <
\beta^T \tilde{v}^- \} = \emptyset$ and because $\cB_{v^-} \cap
\cP \subset \cO \cup \cF$, one obtains $\cB_{v^-} \cap \tilde{\cP}
\subset \cO \cup \tilde{\cF}$. Thus condition (a) for $\tilde{\cP}
\overset{\tilde{\cP}}{\longrightarrow} \tilde{\cF}$   holds. So
the conclusion follows.
\end{proof}

We would like to apply Propositions~\ref{prop1} and \ref{prop2} to
solve the synthesis problem when $\cF$ is not a facet of $\cP$ and
there exists a vertex of $\cF$ satisfying $v_{\ast} \in \cF \cap
\cP^+$. We introduce an important new construct for synthesis of
piecewise affine controllers. Rather than using a subdivision of
$\cP$ we begin the design with a cover of $\cP$, which later will
be refined to a subdivision for control synthesis. A {\em cover}
of $\mathcal{V}$ is a finite collection $\PP=\{\mathcal{P}_1,
\dots, \mathcal{P}_m\}$ of $n$-dimensional polytopes such that the
vertices of each $\mathcal{P}_i$ are drawn from $\mathcal{V}$ and
$\text{conv}(\mathcal{V})$ is the union of $\mathcal{P}_1, \dots,
\mathcal{P}_m$. Informally, a cover is a subdivision except that
the sub-polytopes can intersect on their interiors.

{\bf Cover of $\cP$ with respect to $\cF$:} \label{cover1} \
\begin{enumerate}
\item
Select $v_*$ a vertex of $\cF$ such that $v_{\ast} \in \cF \cap \cP^+$.
\item
Construct any hyperplane that goes through
points $v_-$ and $v_*$, and partitions $\mathcal{P}$ into two $n$-dimensional sub-polytopes
$\mathcal{P}_2$ and $\mathcal{P}_3$.
\item
Define $\cP_1 = \conv(\cF, \cP_2\cap \cP_3)$.
\item
Define the cover $\PP := \{ \cP_1, \cP_2, \cP_3 \}$.
\end{enumerate}

By using this cover, we obtain the following result.
\begin{theorem}
\label{theorem1}
Suppose that $\cF$ is not a facet of $\cP$ and there exists $v_{\ast}$, a vertex of $\cF$,
such that $v_{\ast} \in \cF \cap \cP^+$.
There exists a piecewise affine state feedback that achieves
$\mathcal{P} \overset{\mathcal{P}}{\longrightarrow} \mathcal{F}$
if and only if $\cP \overset{\cP}{\longrightarrow} \cF$ using open-loop control.
\end{theorem}

\begin{proof}
$\mathcal{P}_1$ is an $n$-dimensional polytope in $\mathcal{P}$
for which $\cF$ is a facet. Also, $v_{\ast} \in \cF \cap \cP^+$,
so by Proposition~\ref{prop1}, $\mathcal{P}_1
\overset{\mathcal{P}_1}{\longrightarrow} \mathcal{F}$. Next, let
$\cF_{23} = \cP_2 \cap \cP_3$ and notice that $v^-_{23} = v^-$ and
$v^+_{23} = v_{\ast}$. So by Proposition~\ref{prop2}, $\cP_2
\overset{\cP_2}{\longrightarrow} \cF_{23}  $ and $\cP_3
\overset{\cP_3}{\longrightarrow} \cF_{23} $.

Theorem~\ref{Thm:ControlforPolytope}  gives a piecewise affine control
$u(x)=F_{\sigma_1(x)}x +g_{\sigma_1(x)}$, $x \in \cP_1$, that achieves
$\mathcal{P}_1 \overset{\mathcal{P}_1}{\longrightarrow} \mathcal{F}$.
and it gives
$u(x)=F_{\sigma_2(x)}x +g_{\sigma_2(x)}, x \in \mathcal{P}_2$ and
$u(x)=F_{\sigma_3(x)}x +g_{\sigma_3(x)}, x \in \mathcal{P}_3$, that
achieve $\cP_2 \overset{\cP_2}{\longrightarrow} \cF_{23}$ and
$\cP_3 \overset{\cP_3}{\longrightarrow} \cF_{23}$, respectively. 
Since $\mathcal{F}_{23} \subset \cP_3$, it means that the controllers can drive all the
states not in $\cP_1$ to $\cP_1$. Thus, the
following controller
\[
u(x)=
\left
\{\begin{array}{ll}
F_{\sigma_1(x)}x +g_{\sigma_1(x)} & x \in \cP_1 \\
F_{\sigma_2(x)}x +g_{\sigma_2(x)} & x \in \cP_2 \setminus \cP_1 \\
F_{\sigma_3(x)}x +g_{\sigma_3(x)} & x \in \cP_3 \setminus \cP_1
\end{array}
\right.
\]
achieves $\mathcal{P} \overset{\mathcal{P}}{\longrightarrow} \mathcal{F}$.
\end{proof}

\begin{figure}[!t]
\begin{center}
\psfrag{v6}{$v_6$} \psfrag{v5}{$v_5$} \psfrag{v2}{$v_2$}
\psfrag{v4}{$v_4$} \psfrag{w}{$v^+$} \psfrag{v1}{$v_1$}
\psfrag{vs}{$v_3(v_*)$} \psfrag{f}{$\mathcal{F}$}
\psfrag{e}{$\mathcal{E}$} \psfrag{beta}{$\beta$}
\includegraphics[width= 0.6 \linewidth]{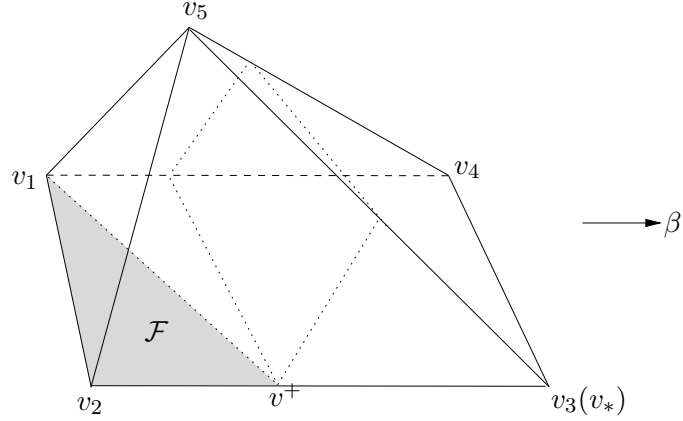}
\caption{$\mathcal{F} \subset \bar{\mathcal{F}}= \conv\{v_1, v_2, v_3, v_4\}$ and all
$v_*$ are in $ \bar{\mathcal{F}}$ but none of them is in $\mathcal{F}$. \label{Fig:Ill2}}
\end{center}
\end{figure}

Finally, we are left with the case when $\mathcal{F}$ is not a facet of $\mathcal{P}$, all
vertices of $\cP$ satisfying Lemma~\ref{Lem:ExtremePoint} are in $\bar{\mathcal{F}}$
but none of them is in $\mathcal{F}$, and moreover there are no vertices of $\cF$ in $\cP^+$. See Figure~\ref{Fig:Ill2}.
Fortunately, this case can be easily handled by our previous results, by observing that $\cF$ and $\cP^+$ are strongly
separated so we can split $\cP$ into a sub-polytope which contains $\cF$ and satisfies Theorem~\ref{theorem1} and another
sub-polytope that does not contain $\cF$ but must be able to reach it. We have the following straightforward extension of
Theorem~\ref{theorem1} and main result of this section.

\begin{theorem}
\label{theorem2}
Suppose that $\cF$ is not a facet of $\cP$.
There exists a piecewise affine state feedback that achieves
$\mathcal{P} \overset{\mathcal{P}}{\longrightarrow} \mathcal{F}$
if and only if $\cP \overset{\cP}{\longrightarrow} \cF$ using open-loop control.
\end{theorem}

\begin{proof}
We only consider the case excluded by
Corollary~\ref{Cor:FnotFacet1} and Theorem~\ref{theorem1} as
described above. Consider the hyperplane $\cB_{v^+}$ that
partitions $\cP$ into two sub-polytopes $\cP_1$ and $\cP_2$, such
that  $\cF \subset \cP_1$ and $v_+$ is a vertex of $\cF$
satisfying $v^+ \in \cF \cap \cP_1^+$ (see Figure~\ref{Fig:Ill2}
for an example). From Theorem~\ref{theorem1}, we have that $\cP_1
\overset{\cP_1}{\longrightarrow} \cF$ and from the assumption $\cP
\overset{\cP}{\longrightarrow} \cF$ and
Theorem~\ref{Thm:NoIntersection} it can be verified that $\cP_2
\overset{\cP_2}{\longrightarrow} \cB_{v^+} \cap \cP$.
\end{proof}

\section{Triangulation with respect to $\cO$}
\label{sec:trianglewrtO}

So far we have studied reachability problems and control synthesis
under the assumption $\ri{\mathcal{P}} \cap \mathcal{O}  =
\emptyset$. In order to solve the general problem when
$\ri{\mathcal{P}} \cap \mathcal{O}  \neq \emptyset$ we want to
partition $\cP$ along $\cO$ and apply the results of the previous
sections. There are two related complications. First, it can happen that when we
split $\cP$ along $\cO$ to form two polytopes, $\cP_1$ and
$\cP_2$, one of the two target sets $\cP_i \cap \cF$, even if not
empty, may no longer be an $(n-1)$-dimensional polytope. Even if
for example $\mathcal{P}_i
\overset{\mathcal{P}_i}{\longrightarrow} \cP_i \cap \cF$ with the
target of dimension less than $(n-1)$, the control synthesis
methods of the previous section do not apply. 
Second, the same lower-dimensional reachability
problem can arise even if we have not already partitioned along $\cO$. Therefore, we
assume in the following that when we say
$\mathcal{P} \overset{\mathcal{P}}{\longrightarrow} \mathcal{F}$,
there does not exist a full-dimensional set of states in
$\mathcal{P}$ that must reach a lower-dimensional (less than $n-1$)
subset in $\cF$ in order to achieve $\mathcal{P}
\overset{\mathcal{P}}{\longrightarrow} \mathcal{F}$. 

Now we would like to propose a partition method which splits $\cP$ along $\cO$ into two
polytopes $\cP_1$ and $\cP_2$. Each subpolytope $\cP_i$ will then have two possible target sets.
One target is the original facet $\cF \cap \cP_i$. A second target is $\cO \cap \cP$. This second
target captures the idea that some trajectories must cross over from one side of $\cO$ to the other
before reaching $\cF$. This means that a new reachability problem must be investigated which involves
two targets. One could try to make a subdivision according to which target the points in $\cP_i$ can reach.
However, this approach will generally require new techniques not already developed in the paper. We illustrate
with an example. 

\begin{example}
Consider the 2D example as in Fig.~\ref{Fig:ce3}.
\begin{figure}[!hbt]
\begin{center}
\psfrag{f}{$\cF_1$}\psfrag{bt}{$\beta$}\psfrag{f2}{$\cF_2$}\psfrag{o}{$\cO$}\psfrag{b}{$\cB$}\psfrag{p}{$\cP$}
\includegraphics[width= 0.6 \linewidth]{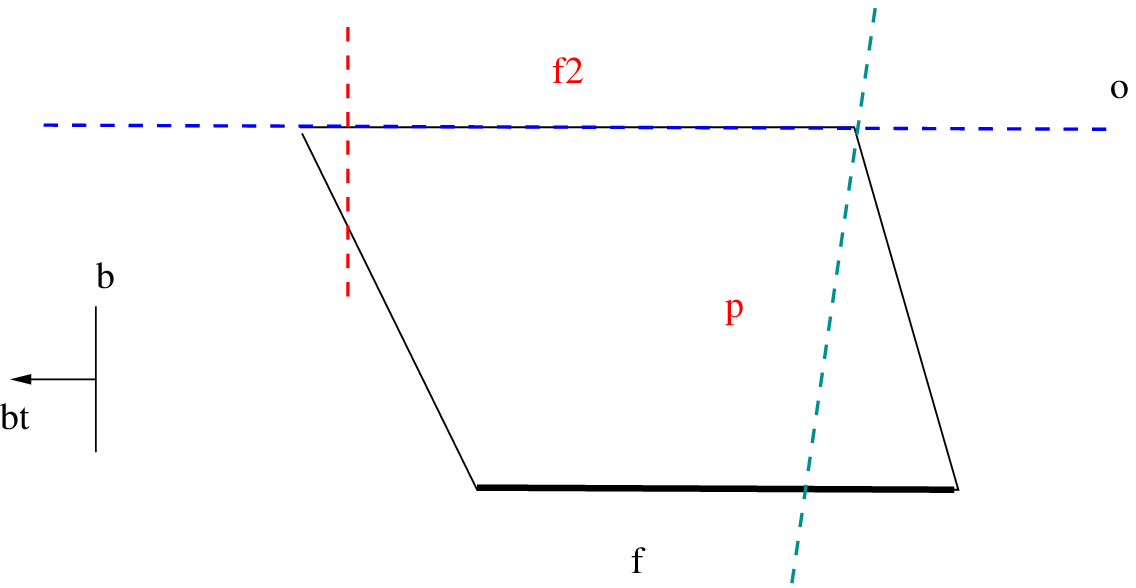}
\caption{ \label{Fig:ce3}}
\end{center}
\end{figure}
Suppose there are two target sets $\cF_1$ and $\cF_2$ where $\cF_2
\subset \cO$. It can be checked that $\cP \to \cF_1 \cup \cF_2$, but
neither $\cP \to \cF_1  $ or $\cP \to   \cF_2$ holds. If we were to apply 
Algorithm 1 to cut off the failure set for reaching $\cF_1$, we would obtain 
the region on the left-side of the (red) dotted line (parallel to $\cB$).
However, the approximate failure set to reach $\cF_1$ cannot reach $\cF_2$, no matter how small is 
$\epsilon$, without crossing into the region that can reach $\cF_1$. Thus, if one insists on
a true subdivision, the reachability problem would not be solvable using our feedback methods.
On the other hand, $\text{Reach}_\epsilon(\cP,\cF_1) $ and $
\text{Reach}_\epsilon(\cP,\cF_2)$ is a cover for $\cP$, where
$\text{Reach}_\epsilon(\cP,\cF_1) $  is the right-side of the red
line and $ \text{Reach}_\epsilon(\cP,\cF_2)$ is the left-side of
the green line.
\end{example}

To most efficiently overcome the issue in the above example, we first subdivide
$\cP$ along $\cO$ and then use a cover in each subpolytope according to two
possible target sets. 

{\bf Cover of $\cP$ with respect to $\cO$:} \label{partition1}
(Let $\epsilon > 0$ be sufficiently small.) \
\begin{enumerate}
\item[(a)] Divide $\cP$ along $\cO$ to obtain $\cP_1$ and $\cP_2$.
\item[(b)] If $\dim(\cP_i \cap \cF) = n-1$, compute $\cQ_{i1} :=
\reach_{\epsilon}(\cP_i,\cP_i \cap \cF)$, $i = 1,2$. Otherwise
$\cQ_{i1} = \emptyset$. \item[(c)] If $\dim(\cP_i \cap \cQ_{j1}) =
n-1$, compute $\cQ_{i2} := \reach_{\epsilon}(\cP_i,\cP_i \cap
\cQ_{j1})$, $i = 1,2$, $j \neq i$. Otherwise $\cQ_{i2} =
\emptyset$. \item[(d)] Define the cover $\PP:=\{ \cQ_{11},
\cQ_{12}, \cQ_{21}, \cQ_{22} \}$.
\end{enumerate}

\begin{theorem}\label{Thm:OInside}
Suppose $\ri{\mathcal{P}} \cap \mathcal{O} \neq \emptyset$. There
exists a piecewise affine state feedback that achieves
$\mathcal{P} \overset{\mathcal{P}}{\longrightarrow} \mathcal{F}$
if and only if $\mathcal{P} \overset{\mathcal{P}}{\longrightarrow}
\mathcal{F}$  using open-loop control.
\end{theorem}

The main idea of the result is that when $\cP$ is partitioned along $\cO$, there are only two types of points in each sub-polytope:
those that reach $\cF$ while remaining in the sub-polytope, or those that cross over to the other polytope to then reach $\cF$.
The proof requires a technical lemma on reachability of two target sets, whose proof is in the Appendix.

\begin{lemma}\label{Lem:Inclusion}
Let $\mathcal{F}_1$ and $\mathcal{F}_2$ be two $(n-1)$-dimensional
polytopes on the boundary of $\mathcal{P}$ but not on a common
hyperplane and assume $\ri{\cP} \cap \cO = \emptyset$. If
$\mathcal{P} \overset{\mathcal{P}}{\longrightarrow} \mathcal{F}_1
\cup \mathcal{F}_2$,  then there exists $\epsilon > 0$
sufficiently small such that $  \reach_{\epsilon}(\cP,\cF_1) \cup
\reach_{\epsilon}(\cP,\cF_2) = \cP$.
\end{lemma}

\begin{proof}[Proof of Theorem~\ref{Thm:OInside}]
($\Longrightarrow$) Obvious. \ ($\Longleftarrow$) \ We use the
notation $\mathcal{P}_i \overset{\mathcal{P}_i}{\Longrightarrow}
\cP_i \cap \cF$ to mean open loop reachability with an
$(n-1)$-dimensional target. We consider two cases. For the first
case, suppose there exists one sub-polytope, say w.l.o.g. $\mathcal{P}_1$,
satisfying $\mathcal{P}_1
\overset{\mathcal{P}_1}{\Longrightarrow} \cP_1 \cap \cF$. If, in
addition, $\mathcal{P}_2 \overset{\mathcal{P}_2}{\Longrightarrow}
\cP_2 \cap \cF$, then we are done. Otherwise, find $\cQ_{21}$ by
the method above. Also compute $\cQ_{22} :=
\reach_{\epsilon}(\cP_2,\cP_2 \cap \cP_1)$. Now we know that if
$\dim(\cP_2 \cap \cF) < n-1$, then $\mathcal{P}
\overset{\mathcal{P}}{\longrightarrow} \mathcal{F}$
 implies $\cQ_{22} = \cP_2$ by our assumption, and we are done.
Instead, if $\dim(\cP_2 \cap \cF) = n-1$ then by
Lemma~\ref{Lem:Inclusion}, $\mathcal{Q}_{21} \cup \cQ_{22}
=\cP_{2}$.

For the second case, suppose no $\mathcal{P}_i \in
\{\mathcal{P}_1, \mathcal{P}_2\}$ satisfies $\mathcal{P}_i
\overset{\mathcal{P}_i}{\Longrightarrow} \cP_i \cap \cF$. Without
loss of generality, suppose $\dim(\cP_1 \cap \cF) = n-1$ and
$\cQ_{11} \neq \emptyset$. Find $\cQ_{21}$, as above. Note that
$\cQ_{21}$ may be empty. Because $\cP
\overset{\cP}{\longrightarrow} \cF$, there exists $\epsilon
> 0$ sufficiently small so that $\dim( \cP_2 \cap \cQ_{11}) =
n-1$ and the states in $\cP_2$ that cannot reach $\cP_2 \cap \cF$
must be able to reach $\cP_2 \cap \cQ_{11}$. Therefore, we have
$\cP_2 \overset{\cP_2}{\longrightarrow} (\cP_2 \cap \cF) \cup
(\cP_2 \cap \cQ_{11})$. Compute $\cQ_{22}$ by the method above.
Now we know that if $\dim(\cP_2 \cap \cF)<n-1$, then $\cP_2
\overset{\cP_2}{\longrightarrow} (\cP_2 \cap \cF) \cup (\cP_2 \cap
\cQ_{11})$ implies $\cQ_{22}=\cP_2$ by our assumption. Instead, if
$\dim(\cP_2 \cap \cF)= n-1$ then by Lemma~\ref{Lem:Inclusion},
$\mathcal{Q}_{21} \cup \cQ_{22} =\cP_{2}$. Repeating the argument
for $\cP_1$, the result is obtained.
\end{proof}

\section{Conclusion}

We have presented methods of triangulation, subdivision, and covers for reachability and control synthesis 
for affine hypersurface systems.  
Some unique features of this work are: (1) We begin with an analysis of open-loop reachability, and we do not impose
what class of controls should be used to implement the reachability specifications. Because of the structure
of hypersurface systems, we then derive that piecewise affine feedbacks are a sufficiently rich class
to solve such problems. (2) We place emphasis on techniques of triangulation and subdivision, guided by the 
the principle that these cannot be performed independently of control synthesis. In particular, we show how the flow conditions
of a system
provide critical information for triangulation, and this can be used to establish greedy dynamic programming algorithms
which are guaranteed to outperform dynamic programming algorithms based on random triangulations of the polytopic state space: 
our algorithm always finds a solution when one exists via open-loop control. (3) We introduce a technique of covers for
forming partitions of the state space. This useful technique overcomes many technical problems with taking subdivisions. 
Fortunately, it naturally leads to synthesis of piecewise affine feedbacks. 

We have concentrated on hypersurface systems because of their simple, well-understood reachable sets. 
To extend our ideas to general systems, a carefully weighed analysis of the tradeoff between the conservatism of 
reach set approximations and complexity of the resulting algorithms must be made. 
Our work points in the direction of keeping the algorithms as simple as possible, by using
the simplest possible partition methods which can guarantee successful termination of numerical procedures. 
Our future work will explore these challenging problems.


\section*{Appendix}  

\renewcommand{\thelemma}{A.\arabic{lemma}}

\subsection{Proof of Proposition~\ref{Thm:NS}}
($\Longrightarrow$)
Assume that there exists an
$\Omega$-invariant set, say $\mathcal{A}$, in $\Omega \setminus
\Omega_f$. For any  $x_0 \in \mathcal{A}$, let $u: t \mapsto u(t)$
be any piecewise continuous function. Then by
Definition~\ref{Def:InvSet} every trajectory in $\Omega$ on an interval is
also in $\mathcal{A}$ on the same time interval.  Furthermore, since  $\mathcal{A} \cap
\Omega_f =\emptyset$ by assumption, it means $x_0 \not \overset{\Omega}{
\longrightarrow} \Omega_f$.

($\Longleftarrow$) Assume it is not true that $\Omega
\overset{\Omega}{\longrightarrow} \Omega_f $. Then $\Omega$ can be
partitioned into two nonempty sets $\Omega'$ and $\Omega''$, where
$\Omega' \overset{\Omega}{\longrightarrow} \Omega_f$ and $\Omega'' \not
\overset{\Omega}{ \longrightarrow} \Omega_f$. It is easily seen that
$\Omega'' \not \overset{\Omega}{ \longrightarrow} \Omega'$ and
$\Omega'=\Omega\setminus \Omega''$. This also immediately implies
that $\Omega''$ is an $\Omega$-invariant set, since otherwise there would exist some
trajectory $\phi_t^u(x_0)$ with $x_0 \in \Omega''$ that reaches $\Omega s\setminus \Omega'' = \Omega'$.
Also $\Omega'' \subset \Omega \setminus \Omega_f$. This completes the proof.

\subsection{Lipschitz Continuity of Marginal Functions}
Let $\mathcal{X}$ and $\mathcal{Y}$ be two sets, $G$ be a set-valued map from
$\mathcal{Y}$ to $\mathcal{X}$ and $f$ be a real-valued function defined on
$\mathcal{X} \times \mathcal{Y}$. We consider the family of maximization
problems
\[
g(y):=\max_{x \in G(y)} f(x,y),
\]
which depend upon the parameter $y$. The function $g$ is called the {\em
marginal function}. A general discussion on continuity properties of marginal
functions can be found in \cite{Aubin84}. Here we focus on the case of linear
affine functions and single out a useful consequence of Lipschitz continuity.
Let
\[
f(x,y)=a^Tx+b \quad \text{and} \quad G(y)=\{x \in \mathcal{P}: c^Tx =y\},
\]
where $a \in \mathbb{R}^n$, $b\in \mathbb{R}$, $c \in \mathbb{R}^n$, and $\mathcal{P}$ is a full dimensional
polytope in $\mathbb{R}^n$. (In another form, $\mathcal{P}$ can be written as $\mathcal{P}=\{x \in \mathbb{R}^n:
Ax \preceq e \}$, where $A \in \mathbb{R}^{m \times n}$, $e \in \mathbb{R}^m$, and $\preceq$ means less or equal
componentwise.) The domain of the marginal function $g$ is given by $\mathcal{D}=\{y \in \mathbb{R}: G(y) \neq
\emptyset\}$.

\begin{lemma}\label{Lem:Lipschitz}
The marginal function $g(y)$ is locally Lipschitz on its domain $\mathcal{D}$.
\end{lemma}

\proof  For any $y_1, y_2 \in \mathcal{D}$, it is clear that
$G(y_1)$ and $G(y_2)$ are lower-dimensional polytopes in
$\mathcal{P}$. Let $v_1, \dots, v_k$ be the vertices of $G(y_1)$.
For each $i=1, \dots, k$, let a point start moving from $v_i$
along the edges of $\mathcal{P}$. It first meets the hyperplane
$c^Tx=y_2$ at a point, denoted by $w_i$. Then, $w_i$ must be a
vertex of $G(y_2)$ (note that $w_i$ and $w_j$ may not be
distinct). The path that the point goes through from $v_i$ to
$w_i$ is composed of either a single edge or joint edges of
$\mathcal{P}$ (see Figure~\ref{Fig:LipschitzProof} for an
illustration in 2D).
\begin{figure}[!t]
\begin{center}
\psfrag{v1}{$v_1$}\psfrag{v2}{$v_2$}\psfrag{w1}{$w_1$}\psfrag{v}{$v$}
\psfrag{w2}{$w_2$}\psfrag{1}{$c^Tx=y_1$}\psfrag{2}{$c^Tx=y_2$}\psfrag{p}{
$\mathcal{P}$}
\includegraphics[width=0.3 \textwidth]{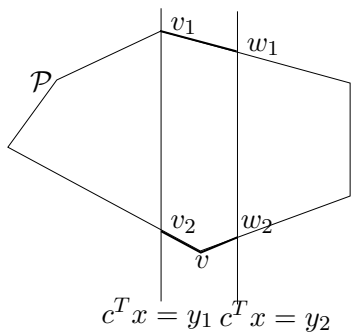}
\caption{An illustration of the proof. \label{Fig:LipschitzProof}}
\end{center}
\end{figure}
Firstly, if it is a single edge of $\mathcal{P}$, this edge can be algebraically
represented by $\{x \in \mathbb{R}^n:A_{i1}x =e_{i1} \text{ and } A_{i2} x\prec
e_{i2}\}$, where $A_{i1} \in \mathbb{R}^{(n-1) \times n}$ and $A_{i2} \in
\mathbb{R}^{(m-n+1) \times n}$ are formed by the columns of $A$ with suitable
order. Since $v_i$ is on the edge and also on the hyperplane $c^Tx=y_1$, it
follows that
\[
v_i =\left[ \begin{array}{c} A_{i1} \\ c^T \end{array} \right]^{-1} \left[
\begin{array}{c} e_{i1} \\ y_1 \end{array} \right].
\]
For the same reason, we have
\[
w_i =\left[ \begin{array}{c} A_{i1} \\ c^T \end{array} \right]^{-1} \left[
\begin{array}{c} e_{i1} \\ y_2 \end{array} \right].
\]
Hence, $\|v_i-w_i\| \leq L_i \|y_1-y_2\|$, where $L_i$ only
depends on $A$ and $c$.  Secondly, if it is composed of joint
edges, without loss of generality, say there are two connected
edges since it has the same argument for the case with more than
two edges. Two edges are connected at a point, say $v$, which lies
between the hyperplanes $c^Tx=y_1$ and $c^Tx=y_2$ (see
Figure~\ref{Fig:LipschitzProof} for an example). Let the parallel
hyperplane going through the point $v$ be $c^Tx=y'$. Thus, $y' \in
[y_1, y_2]$.  By the same argument above, it follows that
$\|v_i-v\| \leq L_i^1 \|y'-y_1\|$ and $\|v-w_i\| \leq L_i^2
\|y_2-y'\|$, where $L_i^1$ and $L_i^2$ depend on $A$ and $c$. Let
$L_i=\max(L_i^1, L_i^2)$. Thus, we have
\[
\|v_i-w_i\| \leq \|v_i-v\|+\|v-w_i\| \leq   L_i^1 \|y'-y_1\|+ L_i^2 \|y_2-y'\|
\leq L_i (\|y'-y_1\|+ \|y_2-y'\|)=L_i\|y_2-y_1\|.
\]

Next, we show that $g(y_1)-g(y_2) \leq L \|y_1-y_2\|$, where $L$ is a constant.
We know that for any $y_1 \in \mathcal{D}$, there exists a $x_1 \in G(y_1)$
satisfying $g(y_1)=f(x_1,y_1)$. On the other hand, the point $x_1$ can be
written as a convex combination of the vertices of $G(y_1)$, i.e.,
$x_1 = \sum_{i=1, \dots, k} \lambda_i v_i$, where $\lambda_i\geq 0$ and
$\sum_{i=1, \dots, k} \lambda_i=1$. Now consider the same convex combination of
points $w_i, i=1, \dots, k$, which is given by $x_2=\sum_{i=1,\dots,k} \lambda_i
w_i$. Notice that $w_i, i=1, \dots, k$ are vertices of $G(y_2)$ as we showed
before, so the point $x_2$ is in $G(y_2)$ and therefore $g(y_2) \geq
f(x_2,y_2)$. Then we deduce that
\[
\begin{array}{ll}
g(y_1)-g(y_2) & \leq f(x_1,y_1) -f(x_2,y_2) \leq \|a\| \|x_1-x_2\| \leq \|a\|
\|\sum _{i=1,\dots, k} \lambda_i (v_i-w_i)\| \\ &\leq \|a\| \sum _{i=1, \dots,
k} \lambda_i \|v_i-w_i\|   \leq \|a\| \sum_{i=1,\dots,k} \lambda_i L_i
\|y_1-y_2\| \leq \|a\| \max_i (L_i) \|y_1-y_2\|.
\end{array}
\]
Recall that $L_i$ depends only on $A$ and $c$. So there is an upper bound $\bar{L}$ only depending on $A$ and
$c$ such that $\bar{L} \geq L_i$ for any $i$. Thus, let $L= \|a\|\bar{L}$ and we obtain  $g(y_1)-g(y_2) \leq L
\|y_1-y_2\|$. Hence, it is locally Lipschitz.
\endproof

\subsection{Proof of Lemma~\ref{lem1}}
Let $G: \mathbb{R} \to 2^{\mathcal{P}}$ be the
set-valued map $G(y)=\{x \in \mathcal{P}: \beta^Tx =y\}$. Its
domain  is $\mathcal{D}=\{y \in \mathbb{R}:G(y) \neq \emptyset\}$.
We define the real-valued function
\[
g(y): =\max\{\beta^T(Ax+a): x \in G(y)\}, \quad y \in \mathcal{D}.
\]
By Lemma~\ref{Lem:Lipschitz}, the function $g(\cdot)$ is locally
Lipschitz. Let $\phi_t(y_0)$ be the solution of $\dot y =g(y)$
with  initial state $y_0$. Since $\beta^T(Ax+a) \leq 0$ for all $x
\in \mathcal{P}$  and $\mathcal{B}_z \cap \mathcal{P} = G(y^*)$,
for some $y^* \in \mathcal{D}$,  we have $g(y^*) \leq 0$. Thus, we
know $\phi_t(y_0) \leq y^*$ for all $t \geq 0$ if $y_0 \leq y^*$.

Now, consider any initial state $x_0 \in \mathcal{H}_z^- \cap \mathcal{P}$ and
any piecewise continuous function $u: t \mapsto u(t)$. Let $x(t), t \in [0,
\bar{T}]$ be the trajectory segment defined in $\mathcal{P}$ with initial
condition $x_0$ and control input $u(t)$. Introduce $\xi(t) =\beta^Tx(t)$, $t
\in [0, \bar{T}]$. Then we have
\[
\dot \xi(t)=\beta^T \dot x(t)=\beta^T(Ax(t)+a+Bu(t))=\beta^T(Ax(t)+a).
\]
Notice that $x(t) \in \mathcal{P}$ and $\beta^T x(t) =\xi(t)$. It
implies that $x(t) \in G(\xi(t))$ for  $t \in [0, \bar{T}]$.
Hence,  we know  $\beta^T(Ax(t)+a) \leq g(\xi(t)) $ (or
equivalently, $ \dot \xi(t) \leq g(\xi(t)) $) from our
construction of $g(\cdot)$. By the Comparison Principle (Theorem
1.4.1, \cite{LakLee69}) it follows that $ \xi(t) \leq
\phi_t(\xi(0))$  for $ t \in [0, \bar{T}]$. Also, $x_0 \in
\mathcal{H}_z^- \cap \mathcal{P}$ implies that $\xi(0) \leq y^*$,
so $\phi_t(\xi(0)) \leq y^*$. Consequently, we obtain $\xi(t) \leq
y^*$, which in turn implies $x(t) \in \mathcal{H}_z^- \cap
\mathcal{P}$ for all $t \in [0, \bar{T}]$, meaning that
$\mathcal{H}_z^- \cap \mathcal{P}$ is $\mathcal{P}$-invariant.
Following along the same lines,  $\ri{\mathcal{H}}_z^- \cap
\mathcal{B}$ is  $\mathcal{P}$-invariant.

\subsection{Proof of Lemma~\ref{lem2}}
For any $x \in l$, let $Ax+a+Bu=\lambda'(y-z)$ where
$\lambda'>0$ is any  constant. Note that $y \in \mathcal{B}_z$
implies $(y-z) \in \text{Im}(B)$, and by convexity we have
$\beta^T(Ax+a)=0$, which implies $(Ax+a) \in \text{Im}(B)$.
Therefore, the above linear equation has a unique solution $u_x$.
Then following along the same lines as for Lemma~\ref{Lem:Line}, it is obtained
that $z \overset{l}{\longrightarrow} y$.

\subsection{Proof of Lemma~\ref{Lem:FaceInvariance}}
If $\beta^T(Ax+a) = 0$ for all $x \in \mathcal{B}_z \cap
\mathcal{P}$, then from Lemma~\ref{lem1} it follows that $\mathcal{H}_{z}^-
\cap \mathcal{P}$ and $\ri{\mathcal{H}}_{z}^- \cap \mathcal{P}$
are $\mathcal{P}$-invariant. On the other hand, rewrite
$\beta^T(Ax+a) = 0$ as $(-\beta)^T(Ax+a) = 0$. Then
$(-\beta)^T(Ax+a) \leq 0$ for all $x \in \mathcal{B}_z \cap
\mathcal{P}$, so again from Lemma~\ref{lem1} we obtain $\mathcal{H}_{z}^+ \cap
\mathcal{P}$ and $\ri{\mathcal{H}}_{z}^+ \cap \mathcal{P}$  are
$\mathcal{P}$-invariant.  By Lemma~\ref{lem:elem},
$(\mathcal{H}_{z}^- \cap \mathcal{P}) \cap (\mathcal{H}_{z}^+ \cap
\mathcal{P})$ and $(\ri{\mathcal{H}}_{z}^- \cap \mathcal{P}) \cup
(\ri{\mathcal{H}}_{z}^+ \cap \mathcal{P})$ are
$\mathcal{P}$-invariant. The former set is exactly $\mathcal{B}_z
\cap \mathcal{P} $ and the latter set is $\mathcal{P} \setminus
\mathcal{B}_z$.

\subsection{Proof of Lemma~\ref{Lem:Line}}
Since $y \in \ri{\mathcal{H}}_z^-$, one obtains
$\beta^T(y-z) <0$. It  implies that the stacked matrix $[-B \ \
(y-z)]$ is of full rank.  Then there is a unique solution $u_x$
and $\lambda_x$ to the linear equation $Ax+a+Bu_x=\lambda_x(y-z)$
for a given point $x \in l$. Moreover, from the assumption $z, y
\not\in \mathcal{O}$, we obtain that $\beta^T(Az+a)<0$,
$\beta^T(Ay+a)<0$ and then by convexity we have $\beta^T(Ax+a)<0$,
$\forall x \in l$. So $\beta^T(Ax+a+Bu_x) <0$. This together with
$\beta^T(y-z) <0$ leads to $\lambda_x >0$, $\forall x \in l$.
Applying $u_x$, the resulting closed-loop system is $\dot x =
\lambda_x (y-z)$. Thus, the trajectory remains in $l$. Moreover,
in the compact set $l$, $\lambda_x$ is bounded away from zero. So
$\beta^T \lambda_x (y-z) <\delta$ for some $\delta<0$, which
implies that the trajectory starting from $z$ reaches $y$ in
finite time.

\subsection{Proof of Proposition~\ref{Prop:ControlforSimplex}\label{Sec:LemmaProof}}
We begin by checking the invariance condition. Let
$w_-$ ($w_+$) be a vertex in $\arg \min\{\beta^T v: v \in \text{vert}(\mathcal{F}_0)\}$ ($\arg \max \{\beta^T v:
v \in \text{vert}(\mathcal{F}_0)\}$, respectively).

First, we consider the case that $v_i \notin \mathcal{O}$. For this, we discuss two situations depending on $\beta^T v_i > \beta^T w_-$ or $\beta^T v_i = \beta^T w_-$. (Note that it is impossible to have $\beta^T v_i < \beta^T w_-$  by Theorem~\ref{Thm:NoIntersection}(a) since $\mathcal{S} \overset{\mathcal{S}}{\longrightarrow} \mathcal{F}_0$.)

(i) If $\beta^T v_i > \beta^T w_-$, then there is a point $p \in \ri{\mathcal{S}}$ such that $\beta^T v_i > \beta^T p$ (or equivalently $\beta^T (p-v_i) <0$). Let
\begin{equation}\label{eq:uNotinO}
Av_i+a+Bu_i = \lambda (p-v_i),
\end{equation}
where $\lambda$ is a scalar to be determined. Writing in a compact form, we have
\[
\left[\begin{array}{cc} -B & (p-v_i)\end{array} \right]\left[\begin{array}{c} u_i \\ \lambda \end{array} \right] =Av_i+a.
\]
Note that $(p-v_i) \notin \mathcal{B}$, so the matrix $[-B \ (p-v_i)] $ is of full rank and therefore the above equation has a unique solution $u_i$ and $\lambda$. Also, notice that $\beta^T (Av_i+a+Bu_i)=\beta^T(Av_i+a) <0$ and that $\beta^T (p-v_i) <0$. Thus, we have $\lambda>0$ from (\ref{eq:uNotinO}). From the definition of simplices, it follows that $h_j \cdot v_i =c_j$ and $h_j \cdot p < c_j$ for any $j\neq i$, where $c_j$ is a constant. This leads to $h_j \cdot (p-v_i) <0$, which further implies
that there exists a $u_i$ attained from (\ref{eq:uNotinO}) satisfying
\begin{equation}\label{eq:RestrictNotinO}
h_j \cdot (Av_i+a+Bu_i) =\lambda h_j \cdot (p-v_i) <0 \text{ for any } j\neq i.
\end{equation}

(ii) If $\beta^T v_i =\beta^T w_-$, then $v_i \neq v_0$ since otherwise it contradicts to $(\mathcal{H}_{w_-}^- \cap \mathcal{S}) \setminus (\mathcal{F}_0 \cup \mathcal{O}) =\emptyset$ inferred from $\mathcal{S} \overset{\mathcal{S}}{\longrightarrow} \mathcal{F}_0$ by Theorem~\ref{Thm:NoIntersection}. Then for this, we claim that $\beta^T v_0 > \beta^T w_-$. (To see this, assume in contrast that $\beta^T v_0 =\beta^T w_-$. Since $v_i \notin \mathcal{O}$, there is a point $y$ on the line segment joining   $v_0$ and $v_i$ and also in a small neighborhood of $v_i$, satisfying $y \notin \mathcal{O}$ and $y \in (\mathcal{H}_{w_-}^- \cap \mathcal{S}) \setminus \mathcal{F}_0$. It contradicts to  $(\mathcal{H}_{w_-}^- \cap \mathcal{S}) \setminus (\mathcal{F}_0 \cup \mathcal{O}) =\emptyset$ again.) Consequently, there is a point $p \in \ri{\mathcal{S}}$ such that $\beta^T v_0 > \beta^T p$. Let
\begin{equation}\label{eq:u-v0-NotinO}
Av_i+a+Bu_i = \lambda (p-v_0),
\end{equation}
where $\lambda$ is a scalar to be determined.  Following along the same lines as above, there exists a $u_i$ attained from (\ref{eq:u-v0-NotinO}) satisfying
\begin{equation}\label{eq:Restrict-NotinO}
h_j \cdot (Av_i+a+Bu_i) =\lambda h_j \cdot (p-v_0) <0 \text{ for any } j\neq 0.
\end{equation}

Second, we consider the case that $v_i \in \mathcal{O}$. For this we discuss two situations depending on $v_0$
(namely, $ \beta^Tw_+ \geq \beta^T v_o \geq \beta^Tw_-$ and $\beta^Tv_0>  \beta^Tw_+$).

(i) If $\beta^Tw_+ \geq \beta^T v_o \geq \beta^Tw_-$, then there is a $p' \in \mathcal{S} \setminus \{v_0\}$ such that $\beta^T v_0 =\beta^Tp'$. Let
\begin{equation}\label{eq:uinO}
Av_i+a+Bu_i =\lambda'(p'-v_0),
\end{equation}
where $\lambda'>0$ is an arbitrary constant. Note that $(p'-v_0) \in \mathcal{B}$ by this choice and that $Av_i+a \in \mathcal{B}$ (due to $v_i \in \mathcal{O}$), so there is a $u_i$ satisfying the equation above. On the other hand, from the definition of simplices, it follows that $h_j \cdot (p'-v_0) \leq 0$ for any $j \neq 0$. Thus, there exists a $u_i$ attained from (\ref{eq:uinO}) satisfying
\begin{equation}\label{eq:RestrictinO}
h_j \cdot (Av_i+a+Bu_i) =\lambda'h_j\cdot (p'-v_0) \leq 0 \text{ for any } j \neq 0.
\end{equation}
In addition, for this $p'$, there has to be a facet $\mathcal{F}_k$ not containing $p'$, where $k \in \{1, \dots, n\}$. Thus, we have $h_k \cdot (p'-v_0) <0 $ and therefore
\begin{equation}\label{eq:FlowinO}
h_k \cdot (Av_i+a+Bu_i) =\lambda'h_k\cdot (p'-v_0) < 0 \text{ for every } v_i \in \mathcal{O}.
\end{equation}

(ii) If $\beta^Tv_0> \beta^Tw_+$, then we claim that $\beta$ together any $n-1$ vectors from $h_1, \dots, h_n$
are linearly independent. (To this end, assume it is not true.
Without loss of generality, we suppose that $\beta$ and $h_1, \dots,
h_{n-1}$ are linearly dependent. Then $\beta$ can be written as
$\beta =\lambda_1 h_1+\cdots+\lambda_{n-1} h_{n-1}$. Thus,
\[
\beta \cdot (v_n-v_0)=\lambda_1 h_1 \cdot  (v_n-v_0)
+\cdots + \lambda_{n-1} h_{n-1} \cdot  (v_n-v_0).
\]
Note that $h_j \cdot
(v_n-v_0) =0$ for any $j=1, \dots, n-1$, so $\beta
\cdot  (v_n-v_0)=0$ and $\beta^T v_n=\beta^Tv_0$, a
contradiction.)
Since $\mathcal{S} \overset{\mathcal{S}}{\longrightarrow } \mathcal{F}_0$ and $\partial \mathcal{S}_{max} =\{v_0\}$ in this case, from Theorem~\ref{Thm:NoIntersection} we have $v_0 \notin \mathcal{O}$.
Let $\delta <0$ be a scalar. Since $\beta$ and $h_1, \dots, h_{i-1},h_{i+1}, \dots, h_n$ are linear independent, there is a unique solution
to the following linear equation
\[\beta \cdot y =0, \
h_j \cdot y =\delta, \quad j=1, \dots,i-1, i+1, \dots, n
\]
Moreover, note that $\beta \cdot y =0$ implies $y \in \mathcal{B}$ and that $v_i \in \mathcal{O}$ implies $Av_i+a \in \mathcal{B}$. So there exists a $u_i $ satisfying $Av_i+a+Bu_i =y$, which further implies that
\begin{equation}\label{eq:Restrict-inO}
h_j \cdot (Av_i+a+Bu_i) =h_j \cdot y =\delta<0 \text{ for any }  j \neq i.
\end{equation}

Thus, it is proved that the invariance condition holds at every vertex,
from (\ref{eq:RestrictNotinO}), (\ref{eq:Restrict-NotinO}), (\ref{eq:RestrictinO}), and (\ref{eq:Restrict-inO}).

Note that once the control inputs $u_0, \dots, u_n$ at corresponding vertices $v_0, \dots, v_n$ are found, an affine control
$u=Fx+g$ can be uniquely constructed by solving the equation
\begin{equation}\label{eq:LAController}
\left[\begin{array}{ccc} u_0 & \cdots & u_n \end{array} \right] =\left[\begin{array}{cc} F & g \end{array} \right] \left[\begin{array}{ccc} v_0 & \cdots & v_n \\ 1 & \cdots & 1 \end{array} \right].
\end{equation}
Now we examine three cases to synthesize the feedback.

First, consider the case when $\beta^T w_+ \geq \beta^T v_o \geq \beta^Tw_-$. Select the control inputs $u_0,
\dots, u_n$ satisfying the invariance condition and construct the affine control $u(x) = F x + g$
from~(\ref{eq:LAController}). With this choice of control, we have shown that (\ref{eq:FlowinO}) also holds for
every vertex $v_i \in \mathcal{O}$. Since $\mathcal{O} \cap \ri{\mathcal{S}}=\emptyset$, one obtains that
$\mathcal{O} \cap \mathcal{S}$ is the convex hull of these vertices in $\mathcal{O}$. By convexity, it follows
from (\ref{eq:FlowinO}) that $h_k \cdot (Ax+a+Bu(x)) <0$ for any $x$ in $\mathcal{O} \cap \mathcal{S}$. Recall
that the possible equilibria of the closed-loop system lie in $\mathcal{O}$. So it implies that no equilibrium
of the closed-loop system is in $\mathcal{S}$. Therefore, by Theorem~\ref{Thm:LinearSimplex}, the affine
control $u(x)=Fx+g$ solves Problem~\ref{Prbm:CTF} and therefore achieves $\mathcal{S}
\overset{\mathcal{S}}{\longrightarrow} \mathcal{F}_0$.

Second, consider the case when $\beta^Tv_0>\beta^Tw_+$ and $\mathcal{F}_0 \not \subset \mathcal{O}$. Select the
control inputs $u_0, \dots, u_n$ satisfying the invariance condition and construct the affine control
$u(x) = F x + g$ from~(\ref{eq:LAController}). We know $v_0 \notin \mathcal{O}$ and there is a vertex $v_k \in
\mathcal{F}_0$ not in $\mathcal{O}$. From (\ref{eq:Restrict-inO}), we have $h_k \cdot (Av_i+a+Bu(v_i)) <0$ for
every $v_i \in \mathcal{O}$ since $k \neq i$. Following along the same lines as above, by
Theorem~\ref{Thm:LinearSimplex}, an affine control $u(x)=Fx+g$ solves Problem~\ref{Prbm:CTF} and therefore
achieves $\mathcal{S} \overset{\mathcal{S}}{\longrightarrow} \mathcal{F}_0$.

Finally, consider the case when $\beta^Tv_0> \beta^Tw_+$ and
$\mathcal{F}_0 \subset \mathcal{O}$. By
Theorem~\ref{Thm:NoIntersection} $\mathcal{B}$ is not parallel to
$\mathcal{O}$, which implies $\beta^T w_+ > \beta^T w_-$. So we
can pick a point $v'$ on the line segment joining $w_-$ and $v_0$
satisfying $\beta^T w_+  > \beta^T v'> \beta^T w_-$. The simplex
$\mathcal{S}$ is then partitioned into two simplices,
$\mathcal{S}_1$ and  $\mathcal{S}_2$, along the hyperplane
containing $v'$ and the vertices in $\text{vert}(\mathcal{F}_0)
\setminus \{w_-\}$. See Figure~\ref{Fig:PWA} for an example.
\begin{figure}[!t]
\begin{center}
\psfrag{1}[][][\scale]{$v_0$} \psfrag{2}[][][\scale]{$w_-$}
\psfrag{1p}[][][\scale]{$v'$}
\psfrag{s1}{$\mathcal{S}_1$}\psfrag{s2}{$\mathcal{S}_2$}
\includegraphics[width= 0.4 \linewidth]{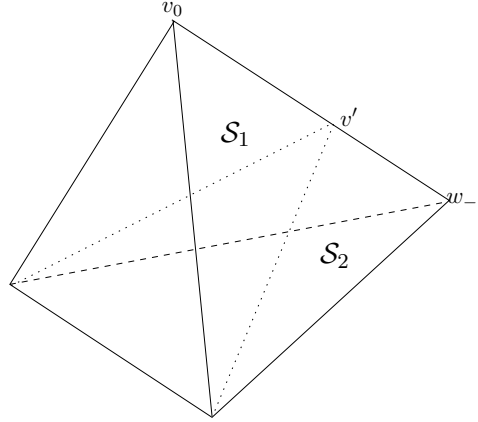}
\caption{Partition of the simplex $\mathcal{S}$. \label{Fig:PWA}}
\end{center}
\end{figure}
Note that in this case $\mathcal{O}$ is the hyperplane containing
$\mathcal{F}_0$, so $v'$ and $v_0$ are not in
$\mathcal{O}$. Let
$\mathcal{F}_0'$ be the common facet of $\mathcal{S}_1$ and $\mathcal{S}_2$.
For $\mathcal{S}_1$, we know $\beta^Tv_0> \max \{\beta^Tv_i: v_i \in \mathcal{F}_0'\}$ and
$\mathcal{F}_0' $ is not in $ \mathcal{O}$.
Hence, from the second case above, there exists an affine feedback $u=F_1x+g_1$ that achieves $\mathcal{S}_1 \overset{\mathcal{S}_1}{\longrightarrow} \mathcal{F}_0'$. For $\mathcal{S}_2$, we have
$\beta^Tw_+ > \beta^Tv' > \beta^Tw_-$. So from the first case above,
there exists an affine feedback $u=F_2x+g_2$ that achieves $\mathcal{S}_2 \overset{\mathcal{S}_2}{\longrightarrow} \mathcal{F}_0$. In total, the feedback
\[
u = \left\{\begin{array}{ll} F_1x+g_1 & \text{ if } x \in \mathcal{S}_1\setminus \mathcal{S}_2, \\
F_2x+g_2 & \text{ if } x \in \mathcal{S}_2
\end{array} \right.
\]
achieves $\mathcal{S} \overset{\mathcal{S}}{\longrightarrow} \mathcal{F}_0$.
\endproof

\subsection{Proof of Lemma~\ref{Lem:Inclusion}}

For $i=1, 2$, let $\mathcal{A}^-_i$ and $\mathcal{A}^+_i$ be the
possible failure sets to reach $\mathcal{F}_i$ defined in
(\ref{eq:FailureSets1}) and (\ref{eq:FailureSets2}), respectively.

We first claim that if $\mathcal{A}^+_i \neq \emptyset$ then
$\cA^+_j = \emptyset$ for $j \neq i$. To see this, suppose that
$\cA^+_j \neq \emptyset$. Then it follows
from~(\ref{eq:FailureSets2}) that $\cA^+_j  =\cP^+$. Also for same
reason, $\mathcal{A}^+_i \neq \emptyset$ implies $\cA^+_i =\cP^+$.
It means $\cP^+$ is a failure set to reach $\cF_1 \cup \cF_2$, a
contradiction to $\mathcal{P}
\overset{\mathcal{P}}{\longrightarrow} \mathcal{F}_1 \cup \cF_2$.

Second, we claim that if $\mathcal{A}^-_i \neq \emptyset$ then
$\cA^-_j = \emptyset$ for $j \neq i$. Suppose instead that both
sets are not empty. Then from~(\ref{eq:FailureSets1}), there is a
point $p \in \arg \min \{\beta^Tx : x \in \cP\}$ that belongs to
both $\cA^-_i$ and $\cA^-_j$. So this point cannot reach $\cF_1
\cup \cF_2$, a contradiction, too.

Third, we claim that $\mathcal{A}^-_i \cap \mathcal{A}^+_j
=\emptyset$ for $i \neq j$. Note that $\mathcal{A}^+_j \subseteq
\cP^+$, so if there is a $p \in \arg \min\{\beta^T x: x \in \cF_i
\} $ such that $p \not \in \cP^+$, then it is clear
from~(\ref{eq:FailureSets1}) that $\mathcal{A}^-_i \cap
\mathcal{A}^+_j =\emptyset$. Instead if for all $p \in \arg
\min\{\beta^T x: x \in \cF_i \}$,  $p  \in \cP^+$, then we know
$\cF_i \subset \cP^+$. So $\cP^+$ is of $(n-1)$-dimension that is
clearly parallel to $\cB$. Notice that $\cO $ is not parallel to
$\cB$ from the controllability assumption. Hence, $\cP^+
\not \subset  \cO \cap \ri{\cH}^+$, which implies $\cA_j^+
=\emptyset$ from~(\ref{eq:FailureSets2}). So the conclusion
follows.

Now we come to prove that $  \reach_{\epsilon}(\cP,\cF_1) \cup
\reach_{\epsilon}(\cP,\cF_2) = \cP$. Let
$\mathcal{A}_{\epsilon_-}^i$ and $\mathcal{A}_{\epsilon_+}^i$
($i=1, 2$) be the over-approximations of $\mathcal{A}^-_i$ and
$\mathcal{A}^+_i$ obtained by applying
Algorithm~\ref{Alg:PracticalCut}. Consider a point $x \in
\mathcal{P} \setminus \reach_{\epsilon}(\cP,\cF_1)$ if it exists.
Then $x$ is either in $\mathcal{A}_{\epsilon_-}^1$ or in
$\mathcal{A}_{\epsilon_+}^1$. Consider the first case when $x \in
\mathcal{A}_{\epsilon_-}^1$. That means, $\mathcal{A}^-_1 \neq
\emptyset$. Thus by our first claim, we get $\mathcal{A}^-_2
=\emptyset$. Moreover, by our third claim that $\mathcal{A}^-_1
\cap \mathcal{A}^+_2 =\emptyset$, we know for sufficiently small
$\epsilon$, $\mathcal{A}_{\epsilon_-}^1 \cap
\mathcal{A}_{\epsilon_+}^2 =\emptyset$. Thus $x \not \in
\mathcal{A}_{\epsilon_+}^2$ and it must be in
$\reach_{\epsilon}(\cP,\cF_2)$. Consider now the second case when
$x \in \mathcal{A}_{\epsilon_+}^1$. That means, $\mathcal{A}^+_1
\neq \emptyset$. Then by our second claim, we obtain that
$\mathcal{A}^+_2 = \emptyset$. Moreover, since $\mathcal{A}^+_1
\cap \mathcal{A}^-_2 =\emptyset$, then by the same argument as
above, we know the point $x$ has to be in
$\reach_{\epsilon}(\cP,\cF_2)$.


\bibliographystyle{IEEEtranS}

\begin{thebibliography}{99}

\bibitem{Aubin84}
J.P. Aubin and A. Cellina.
{\em Differential Inclusions: Set-Valued Maps and Viability Theory}.
Springer-Verlag, 1984.

\bibitem{Aub01}
J. P. Aubin.
Viability kernels and capture basins of sets under differential inclusions.
{\em SIAM Journal on Control and Optimization}.
vol. 40, no. 3, pp. 853--881, 2001.

\bibitem{BARIC}
M. Baric, P. Grieder, M. Baotic, and M. Morari.
An efficient algorithm for optimal control of PWA systems with polyhedral performance indices
{\em Automatica}.
vol. 44, Issue 1, January 2008, pp. 296-301.

\bibitem{BEMPORAD}
A. Bemporad, G. Ferrari-Trecate, and M. Morari.
Observability and Controllability of Piecewise Affine and Hybrid Systems.
{\em IEEE Transactions on Automatic Control}.
vol. 45, no. 10, October 2000.

\bibitem{BLANCHINI}
F. Blanchini and S. Miani.
{\em Set-Theoretic Methods in Control}.
Birkhauser, 2008.

\bibitem{BLANCHINI2}
F. Blanchini and F.A. Pellegrino.
Relatively optimal control: A static piecewise affine solution.
{\em SIAM Journal on Control and Optimization}
vol. 46, issue 2, p585--603, 2007.

\bibitem{BOSCAIN}
U. Boscain.
Stability of Planar Switched Systems: the linear Single Input Case.
{\em SIAM Journal on Control and Optimization}.
vol. 41, no. 1, pp. 89--112, 2002.

\bibitem{FARCOT}
E. Farcot, J.L. Gouze.
A mathematical framework for the control of piecewise-affine models of gene networks
{\em Automatica}.
vol. 44, Issue 9, September, 2008, pp. 2326-2332.

\bibitem{GEYER}
T. Geyer, F.D. Torrisi and M. Morari.
Optimal complexity reduction of polyhedral piecewise affine systems
{\em Automatica}
vol. 44, issue 7, July, 2008, pp. 1728-1740.

\bibitem{JacJos97}
J. E. Goodman and J. O'Rourke, Eds..
{\em Handbook of Discrete and Computational Geometry}.
CRC Press, 1997.

\bibitem{HVS01}
L.C.G.J.M. Habets and J.H. van Schuppen.
Control of piecewise-linear hybrid systems on simplices and rectangles,
in: M.D. Di Benedetto and A.L. Sangiovanni-Vincentelli (Eds.)
{\it Hybrid Systems: Computation and Control, Lecture Notes in Computer Science}.
Springer Verlag, vol. 2034, pp. 261--274, 2001.

\bibitem{HVS04}
L.C.G.J.M. Habets and J.H. van Schuppen.
A control problem for affine dynamical systems on a full-dimensional polytope.
{\it Automatica}. no. 40, pp. 21--35, 2004.

\bibitem{HVS06}
L.C.G.J.M. Habets, P.J. Collins, and J.H. van Schuppen.
Reachability and control synthesis for piecewise-affine hybrid systems on simplices.
{\em IEEE Trans. Automatic Control} no. 51, pp. 938--948, 2006.

\bibitem{LakLee69}
V. Lakshmikantham and S. Leela.
{\em Differential and Integral Inequalities: Theory and Applications}.
Academic Press, 1969.

\bibitem{LB06}
Z. Lin and M.E. Broucke.
Resolving control to facet problems for affine hypersurface systems on simplices.
{\em IEEE Conference on Decision and Control} (CDC '06).
December 2006.

\bibitem{LB07}
Z. Lin and M.E. Broucke.
Reachability and control of affine hypersurface systems on polytopes.
{\em IEEE Conference on Decision and Control} (CDC '07).
December 2007.

\bibitem{NenFre02}
G. Nenninger, G. Frehse, and V. Krebs.
Reachability analysis and control of a special class of hybrid systems,
in {\em Modelling, Analysis and Design of Hybrid Systems}.
S. Engell, G. Frehse, and E. Schnieder, Eds.
Springer-Verlag, 2002, pp. 173--192.

\bibitem{QuiVel98}
M. Quincampoix and V. Veliov.
Viability with a target: theory and applications, in
{\em Applications of Mathematics in Engineering}.
Heron Press, pp. 47--54, 1998.

\bibitem{ROLL}
J. Roll, A. Bemporad, L. Ljung.
Identification of piecewise affine systems via mixed-integer programming
{\em Automatica}.
vol. 40, issue 1, January 2004, pp. 37-50.

\bibitem{RB06}
B. Roszak and M. E. Broucke.
Necessary and sufficient conditions for reachability on a simplex.
{\em Automatica}. vol. 42, no. 11, pp. 1913--1918, November 2006.

\bibitem{RB07}
B. Roszak and M. E. Broucke.
Reachability of a set of facets for linear affine systems with $n-1$ inputs.
{\em IEEE Transactions on Automatic Control}.
vol. 52, no. 2, pp. 359-364, February 2007.
\end{thebibliography}

\end{document}